\documentclass[11pt]{article}
\usepackage{amssymb, amsmath, enumerate}
\usepackage{graphicx}
\usepackage{amsthm}
\usepackage{array}
\usepackage{mathtools}
\usepackage[parfill]{parskip}
\usepackage{mathrsfs}
\usepackage{caption}
\usepackage{subcaption}
\usepackage{bm}
\usepackage{amsfonts,amscd}
\usepackage{courier}
\usepackage[]{units}
\usepackage{listings}
\usepackage{multicol}
\usepackage{tcolorbox}
\usepackage{float}
\usepackage[square, comma, sort&compress, numbers]{natbib} 
\usepackage{physics}
\usepackage{tikz}
\usepackage{algorithmicx} 

\usepackage{svg}
\usepackage[ruled,vlined]{algorithm2e}
\usepackage{bibentry}
\usepackage{booktabs}
\usepackage[noend]{algpseudocode}
\usepackage{relsize}
\usepackage{booktabs} 
\usepackage[top=1in, bottom=1in, left=1in, right=1in]{geometry}

\usepackage[linktocpage,colorlinks,linkcolor=blue,anchorcolor=blue,citecolor=blue,urlcolor=blue,pagebackref]{hyperref}
\newtheorem{theorem}{Theorem}[section]
\newtheorem{lemma}{Lemma}[section]
\newtheorem{corollary}{Corollary}[section]
\newtheorem{proposition}{Proposition}[section]

\newtheorem{remark}{Remark}[section]

\newtheorem{definition}{Definition}[section]

\newcommand{\bx}{\boldsymbol{x}}

\newcommand{\ba}{\mbox{\boldmath{$a$}}}
\newcommand{\bi}{\mbox{\boldmath{$i$}}}
\newcommand{\bj}{\mbox{\boldmath{$j$}}}
\newcommand{\bk}{\mbox{\boldmath{$k$}}}
\newcommand{\bq}{\mbox{\boldmath{$q$}}}
\newcommand{\bp}{\mbox{\boldmath{$p$}}}

\newcommand{\bA}{\mbox{\boldmath{$A$}}}
\newcommand{\bF}{\mbox{\boldmath{$F$}}}
\newcommand{\bH}{\mbox{\boldmath{$H$}}}

\newcommand{\mR}{\mathbb{R}}

\newcommand{\mH}{\mathbb{H}}

\newcommand{\re}{\mbox{\rm Re~}}
\newcommand{\im}{\mbox{\rm Im~}}

\title{On Approximation Algorithms for Commutative Quaternion Polynomial Optimization}
\date{}
\author{
{Chang He} \thanks{Research Institute for Interdisciplinary Sciences, Shanghai University of Finance and Economics, 
             Shanghai 200433, People’s Republic of China. \texttt{ischanghe@gmail.com}}
\and 
{Bo Jiang} \thanks{Research Institute for Interdisciplinary Sciences, Shanghai University of Finance and Economics, 
             Shanghai 200433, People’s Republic of China. \texttt{isyebojiang@gmail.com}}
\and
{Hongye Wang} \thanks{Research Institute for Interdisciplinary Sciences, Shanghai University of Finance and Economics, 
              Shanghai 200433, People’s Republic of China. \texttt{ishongyewang@gmail.com}}
\and
{Xihua Zhu} \thanks{Faculty of Business Information, Shanghai Business School, Shanghai 200235, People’s Republic of China. \texttt{simexihuazhu@163.com}}
}

\begin{document}




\maketitle

\begin{abstract}
Quaternion optimization has attracted significant interest due to its broad applications, including color face recognition, video compression, and signal processing. 
Despite the growing literature on quadratic and matrix quaternion optimization, to the best of our knowledge, the study on quaternion polynomial optimization still remains blank. In this paper, we introduce the first investigation into this fundamental problem, and focus on the sphere-constrained homogeneous polynomial optimization over the commutative quaternion domain, which includes the best rank-one tensor approximation as a special case.
Our study proposes a polynomial-time randomized approximation algorithm that employs tensor relaxation and random sampling techniques to tackle this problem. Theoretically, we prove an approximation ratio for the algorithm providing a worst-case performance guarantee.
\end{abstract}


\section{Introduction}

Quaternions \cite{parcollet2019quaternion,hamilton1853lectures}, renowned for the powerful representation capabilities, have been extensively explored in various fields \cite{liu2023quaternion,wang2023quaternion,miron2023quaternions,fu2020dual,flamant2021general,flamant2019time,flamant2018complete,heller2014hand}, along with the development of solution methods \cite{qi2020quaternion,qi2022quaternion,chen2020low,chen2022color,chen2024regularization,qi2023standard,hadi2024se,ling2022minimax,cui2024power,lyu2024randomized,jia2024new,pan2023block,chen2023phase,he2023eigenvalues,he2025eigenvalues}. However, the non-commutative nature of quaternions render operations such as Fourier transformation \cite{hitzer2007quaternion}, convolution, and correlation tedious to implement, limiting further applications. To overcome this limitation, Segre \cite{segre1892real} proposed the \textit{commutative quaternions}
$$\mathbb{H}=\{\boldsymbol{q}=q_0+q_1\bi+q_2\bj+q_3\bk: q_0,q_1,q_2,q_3\in \mR\},
$$
where the imaginary units $\bi, \bj, \bk$ satisfy the following multiplication rules:
$$\bi^2=\bk^2=-1,\ \bj^2=1,\ \bi\bj\bk=-1,\ \bi\bj=\bj\bi=\bk,\ \bj\bk=\bk\bj=\bi, \ \bk\bi=\bi\bk=-\bj.
$$
This formulation endows commutative quaternions with distinctive properties, including zero-divisors and isotropic elements, and most notably, allows for commutative multiplication. These advantages have motivated the application of commutative quaternions in various fields, including signal processing \cite{pei2004commutative,borio2023bicomplex,grigoryan2022commutative,he2023eigenvalues,he2025eigenvalues}, neural networks \cite{isokawa2010commutative,isokawa2012quaternionic,xia2020some,takahashi2022remarks,kobayashi2018twin,kobayashi2020hopfield}, and others \cite{zhang2020algebraic,atali2023new,hitzer2021quaternion,he2022decomposition}. Concurrently, fundamental theory of commutative quaternion has developed a lot during recent years \cite{kosal2014commutative,szynal2022generalized,ding2024algebraic,chen2023eta,zhang2024singular,kosal2019universal,kosal2017some}. The increasing number of applications modeled by commutative quaternion and the studies on associated problems call for a deeper understanding of the commutative quaternion structure. In this paper, we focus on a specific type of problem: \textit{polynomial optimization} over the commutative quaternion domain.

Polynomial optimization \cite{anjos2011handbook,li2012approximation,jiang2013polynomial} is a popular research topic in mathematical optimization due to its broad applications and elegant theoretical results. As most polynomial optimization problems are NP-hard, on the front of approximate solutions, the design of related approximation algorithms with performance guarantees is also appealing \cite{yang2022approximation,mao2022several,mao2022best}. Currently, the polynomial optimization models under consideration are mostly in the domain of real and complex numbers. In the real domain, Luo and Zhang \cite{luo2010semidefinite} conducted approximation algorithms for quartic polynomial optimization problems with provable worst-case approximation ratios. Then He et al. \cite{he2009general} extended the techniques for any degree polynomials. Furthermore, the polynomial optimization problem with different constraints was also widely studied, to name a few, \cite{he2010approximation,he2013approximation,he2014probability,he2015inhomogeneous,yang2019efficiently}. When it comes to the complex domain, So et al. \cite{so2007approximating} presented a unified analysis for a class of discrete and continuous quadratic optimization problems in the complex Hermitian form. Later, Huang and Zhang \cite{huang2010approximation} developed an approximation algorithm for quadratic bilinear form complex optimization models under the unity constraint. For general degree polynomial optimization problems, Jiang et al. \cite{jiang2014approximation} proposed several approximation algorithms with worst-case approximation ratio guarantees.

Going beyond real and complex domains, a natural curiosity arises about generalizing these methodologies into the quaternion domain, and we aim to take \textit{the first step} along this direction. Formally, the problem of our interest is the following polynomial optimization model over the commutative quaternion domain
\begin{equation}
    \begin{aligned}
        (P) \quad & \max \ \re \boldsymbol{H}(\bx) \\
         &\operatorname{s.t.} \quad \bx \in \mathbf{S}^n. \notag
    \end{aligned}
\end{equation}
Here, $\bH(\bx)$ is the homogeneous polynomial defined in \eqref{eq.homogeneous polynominal}, and $\mathbf{S}^n$ is the quaternion spherical constraint: $\mathbf{S}^n=\{\bx\in\mH^n: \Vert \bx\Vert=1\}$. Similar to its counterparts in real complex cases, the optimization problem $(P)$ is NP-hard in general \cite{he2010approximation,jiang2014approximation}. Therefore, the development of a randomized approximation algorithm is a worthwhile study. To this end, we summarize the main contributions of this paper. We first extend the concept of multilinear forms and homogeneous polynomials to the commutative quaternion domain, alongside the corresponding quaternion tensors. Moreover, we demonstrate that the problems that this model can accommodate a fundamental problem in linear algebra: the best rank-one approximation of a commutative quaternion tensor. Using a novel quaternion probabilistic inequality, we introduce a randomized approximation algorithm for solving the homogeneous polynomial optimization problem $(P)$ through tensor relaxation and random sampling techniques. This algorithm yields an approximated solution within polynomial time. In contrast to widely adopted real structure-preserving methods \cite{li2016real,li2017real,jia2018new}, our approach is \textit{intrinsic}, as it performs operations directly on quaternions. 

The organization of the paper is as follows. In Section \ref{sec.preliminary}, we present the notations and definitions used throughout this work, as well as the example that motivates our model. Section \ref{section.inequality} focuses on proving a probability inequality, which is crucial for establishing the approximation ratio of our randomized algorithm (Algorithm \ref{Alg: Randomized Alg 1}) proposed later in Section \ref{section.algorithm}. A bridge between multilinear forms and homogeneous polynomials is demonstrated in Section \ref{section.homogenous}, and thus the approximation ratio of  Algorithm \ref{Alg: Randomized Alg 2} concerning problem $(P)$ can be obtained. Section \ref{sec.numerical} provides numerical experiments that validate our theoretical findings (Theorem \ref{tm.approximation theorem}), using a specially constructed problem equipped with an explicit upper bound.
\section{Preliminaries and the Motivating Example}\label{sec.preliminary}
In this section, we introduce some basic algebraic operations in the commutative quaternion domain for scalars, vectors, and matrices. For real vectors and matrices, we denote the 2-norm by $\|\cdot\|$, and the $\infty$-norm by $\|\cdot\|_\infty$.

\subsection{Commutative Quaternion Operations}
Throughout the paper, we denote elements in the commutative quaternion domain using bold fonts, such as $\bq$, $\bA$, and $\bF$, while elements in the real domain are represented in regular font (e.g. $q$, $A$, and $F$). For a commutative quaternion
\begin{equation*}
    \boldsymbol{q}=q_0+q_1\bi+q_2\bj+q_3\bk \in \mathbb{H},
\end{equation*}
the real and imaginary parts of $\bq$ are denoted as $\re (\bq) = q_0$ and $\im (\bq) = q_1\bi+q_2\bj+q_3\bk$, respectively. 
We define the conjugate of a commutative quaternion $\boldsymbol{q}$ as follows:
    \begin{equation*}
        \overline{\boldsymbol{q}}=q_0-q_1\bi + q_2 \bj -q_3\bk,
    \end{equation*}
which is known as the first kind of principal conjugation for commutative quaternions \cite{kosal2014commutative}. This notion can be used to define the magnitude of the quaternion, which is given by:
\begin{equation*}
    |\bq|=\sqrt{\re(q\cdot \bar{q})}=\sqrt{q_0^2+q_1^2+q_2^2+q_3^2}.
\end{equation*}
With a slight abuse of notation, a commutative quaternion vector $\bq \in \mH^n$ is also written as
\begin{equation*}
    \boldsymbol{q}=q_0+q_1\bi+q_2\bj+q_3\bk, 
\end{equation*}
where $q_0, q_1, q_2, q_3 \in \mR^n$ are the components of $\bq$. For a commutative quaternion vector $\bq$, $\bq^{\top}$ denotes the transpose of $\bq$, and $\bq^{H}=(\overline{\bq})^{\top} = \overline{(\bq^{\top})} = q_0^\top-q_1^\top \bi+q_2^\top \bj-q_3^\top \bk.$ denotes its conjugate transpose.
The inner product of the two quaternion vectors $\bq, \bp \in \mH^n$ is defined as
$$
\bq \bullet \bp = \re(\bq^H \bp) = q_0^\top p_0+q_1^\top p_1+q_2^\top p_2+q_3^\top p_3.
$$
Note that this definition specifically takes the real part of the quaternion product $\bq^H \bp$ to ensure the inner product is real-valued. This property allows for the definition of the vector norm as follows:
$$
\|\bq\| = (\bq \bullet \bq)^{\frac{1}{2}} =\left(\sum_{i=0}^3 q_i^\top q_i\right)^{\frac{1}{2}}.
$$
We immediately obtain $\re(\boldsymbol{q}^\top \boldsymbol{p})\le \|\boldsymbol{p}\| \cdot \|\boldsymbol{q}\|$ for any $\boldsymbol{q}, \boldsymbol{p}\in \mH^n$. For any commutative quaternion matrix $\boldsymbol{A} \in \mH^{m\times n}$, it can be expressed as
\begin{equation*}
    \boldsymbol{A} = A_0+A_1\bi+A_2\bj+A_3\bk,
\end{equation*}
with $A_0, A_1, A_2, A_3 \in \mR^{m\times n}$. The transpose and the conjugate transpose of $\boldsymbol{A}$ are $\boldsymbol{A}^{\top}$ and $\boldsymbol{A}^H=(\overline{\boldsymbol{A}})^{\top} = \overline{(\boldsymbol{A}^{\top})}=A_0^\top - A_1^\top \bi + A_2^\top \bj - A_3^\top \bk$, respectively. Following a similar manner to vectors, the inner product of two quaternion matrices is defined as
$$
\boldsymbol{A} \bullet \boldsymbol{B}=\operatorname{Tr}(\re(\boldsymbol{A}^H \boldsymbol{B}))=\operatorname{Tr}\left(A_0^{\top} B_0+A_1^{\top} B_1+A_2^{\top} B_2+A_3^{\top} B_3\right),
$$
where ``Tr" is the trace of a matrix. Hence, the norm of $\boldsymbol{A}$ is defined by
$$
\|\boldsymbol{A}\|=\sqrt{\boldsymbol{A} \bullet \boldsymbol{A}}=\sqrt{\operatorname{Tr}\left(A_0^{\top} A_0+A_1^{\top} A_1+A_2^{\top} A_2+A_3^{\top} A_3\right)}.
$$

Similarly, for any commutative quaternion tensor $\boldsymbol{\mathcal{T}} \in \mH^{n_1\times \cdots\times n_d}$, it can be expressed as
\begin{equation*}
    \boldsymbol{\mathcal{T}} = \mathcal{T}_0+\mathcal{T}_1\bi+\mathcal{T}_2\bj+\mathcal{T}_3\bk,
\end{equation*}
with $\mathcal{T}_0, \mathcal{T}_1, \mathcal{T}_2, \mathcal{T}_3 \in \mR^{n_1\times \cdots\times n_d}$. Motivated by inner product of real tensors \cite{kolda2009tensor}, we define the inner product of two quaternion tensors as
$$
\boldsymbol{\mathcal{T}} \bullet \boldsymbol{\mathcal{K}}=\langle \mathcal{T}_0,\mathcal{K}_0\rangle+\langle \mathcal{T}_1,\mathcal{K}_1\rangle+\langle \mathcal{T}_2,\mathcal{K}_2\rangle+\langle \mathcal{T}_3,\mathcal{K}_3\rangle.
$$
Hence, the norm of $\boldsymbol{\mathcal{T}}$ is defined by
$$
\|\boldsymbol{\mathcal{T}}\|=\sqrt{\boldsymbol{\mathcal{T}} \bullet \boldsymbol{\mathcal{T}}}=\sqrt{\|\mathcal{T}_0\|^2+\|\mathcal{T}_1\|^2+\|\mathcal{T}_2\|^2+\|\mathcal{T}_3\|^2},
$$
where $\|\cdot\|$ is the norm of real tensors \cite{kolda2009tensor}. We also call the tensor $\boldsymbol{\mathcal{T}} \in \mH^{n_1\times \cdots\times n_d}$ who has $d$ demensions as $d$th-order commutative quaternion tensor. Specifically, a commutative quaternion tensor $\boldsymbol{\mathcal{T}}\in \mathbb{H}^{n^d}$ is super-symmetric if its entries $\boldsymbol{\mathcal{T}}_{i_1 i_2 \cdots i_d}$ are invariant under permutations of the indices ${i_1,i_2,\cdots,i_d}$, that is,
    $$
    \boldsymbol{\mathcal{T}}_{i_1 i_2 \cdots i_d}=\frac{a_{i_1i_2\cdots i_d}}{\vert \Pi(i_1i_2\cdots i_d) \vert},\forall 1\le i_1\le i_2\le\cdots\le i_d\le n,
    $$
    where $\Pi(i_1i_2\cdots i_d)$ is the set of all permutations of the indices $\{i_1,i_2,\cdots,i_d\}$, $a_{i_1i_2\cdots i_d}\in \mathbb{H}$ is a constant commutative quaternion and $|\cdot|$ denotes cardinality of the set.

\subsection{Homogeneous Polynomial in Quaternion Domain}
A multivariate polynomial $f(\bx)$ over commutative quaternion domain is a function of variable $\bx \in \mH^n$ whose coefficients are commutative quaternion, e.g. $f(\bx_1, \bx_2) = (\bi + \bj)\bx_1 + (1 - \bk)\bx_2$. Specifically, we focus on a general $n$-dimensional $d$-th degree homogeneous polynomial function $\boldsymbol{H}(\bx)$, which can be explicitly expressed as a sum of commutative quaternion monomials:
\begin{equation}\label{eq.homogeneous polynominal}
\boldsymbol{H}(\bx)=\sum_{1\le i_1\le i_2\le\cdots\le i_d\le n}\ba_{i_1i_2\cdots i_d}\bx_{i_1}\bx_{i_2}\cdots \bx_{i_d},~ \text{where}~\ba_{i_1i_2\cdots i_d}  \in\mH.
\end{equation}
Motivated by the CP rank \cite{kolda2009tensor} in the real domain, we introduce the concept of the commutative quaternion tensor rank as follows. 
\begin{definition}[the rank of commutative quaternion tensor]
    The rank of a commutative quaternion tensor $\boldsymbol{\mathcal{F}} \in \mH^{n_1\times\cdots\times n_d}$ is the smallest optimal solution of the following optimization problem:
    \begin{equation*}
        \rank(\boldsymbol{\mathcal{F}})=\min \left\{r \in \mathbb{R} : \boldsymbol{\mathcal{F}} = \sum_{i=1}^r\boldsymbol{x}^1_i\otimes \cdots \otimes \boldsymbol{x}^d_i \right\},
    \end{equation*}
    where $\bx_i^j\in \mH^{n_{j}}$, $i=1,2\cdots,r$, $j=1,2,\cdots,d$, and $\otimes$ represents the outer product, $\boldsymbol{x}_i^1\otimes \cdots \otimes \boldsymbol{x}_i^d = \mathcal{T}\in \mH^{n_1\times\cdots\times n_d}$, $(\mathcal{T})_{k_1,\cdots,k_d} = (x_i^1)_{k_1}\cdots(x_i^d)_{k_d}$.
    
\end{definition}
To illustrate the aforementioned definition, let us consider an example:

when $\bx^1_{i_1i_2\cdots i_d} = (0, \ldots, \underbrace{\boldsymbol{\mathcal{F}}_{i_1i_2\cdots i_d}}_{i_1-th}, \ldots, 0)^\top \in \mH^{n_1}$, $\bx^j_{i_1i_2\cdots i_d} = (0, \ldots, \underbrace{1}_{i_j-th}, \ldots, 0)^\top\in \mH^{n_j}$, $ j = 2, \cdots, d$ and $i_k=1,2,\cdots,n_k$, $k =1, 2, \cdots, d$, it means that $\sum_{i_k\in \{1,2,\cdots,n_k\}},k =1, 2, \cdots,d\boldsymbol{x}^1_{i_1i_2\cdots i_d}\otimes \cdots \otimes \boldsymbol{x}^d_{i_1i_2\cdots i_d}$ becomes a feasible solution to the above minimization problem, indicating the existence of an upper bound $n_1\times n_2\times\cdots\times n_d$. Hence, the definition is well-defined. Correspondingly, 
a tensor $\boldsymbol{\mathcal{F}}\in \mH^{n_1\times\cdots\times n_d}$ is said to be rank-one if there exists $\boldsymbol{x}^1,\boldsymbol{x}^2,\cdots,\boldsymbol{x}^d$ such that
    $$
    \boldsymbol{\mathcal{F}}=\boldsymbol{x}^1\otimes \boldsymbol{x}^2\otimes  \cdots \otimes \boldsymbol{x}^d.
    $$

Given a $d$th-order commutative quaternion tensor $\boldsymbol{\mathcal{F}} \in \mathbb{H}^{n_{1}\times \ldots \times n_{d}}$, the associated multilinear form is defined as
\begin{equation}
\boldsymbol{F}(\boldsymbol{x}^1, \cdots, \boldsymbol{x}^d) = \sum_{i_1=1}^{n_1} \ldots \sum_{i_d=1}^{n_d} \boldsymbol{\mathcal{F}}_{i_1 \ldots i_d} \boldsymbol{x}_{i_1}^1 \ldots \boldsymbol{x}_{i_d}^d,\nonumber
\end{equation}
where $\boldsymbol{x}^k \in \mathbb{H}^{n_k}$ for $k=1, \ldots, d$. Besides, we use the notation 
$$F(\boldsymbol{x}^1,\cdots,\boldsymbol{x}^{d-t},\bullet,\boldsymbol{x}^{d-t+2},\cdots,\boldsymbol{x}^d)\in \mH^{n_{d-t+1}}$$
to denote a commutative quaternion vector which satisfies
$$
    \big(F(\boldsymbol{x}^1,\cdots,\boldsymbol{x}^{d-t},\bullet,\boldsymbol{x}^{d-t+2},\cdots,\boldsymbol{x}^d)\big)^\top \boldsymbol{y} = F(\boldsymbol{x}^1,\cdots,\boldsymbol{x}^{d-t},\boldsymbol{y},\boldsymbol{x}^{d-t+2},\cdots,\boldsymbol{x}^d)
$$
for any $\boldsymbol{y}\in \mH^{n_{d-t+1}}$. 

For the super-symmetric tensor, the homogeneous polynomial $\bH$ is derived from the multilinear form by setting $\boldsymbol{x}^1=\boldsymbol{x}^2=\cdots=\boldsymbol{x}^d$. 
The connection between homogeneous polynomials and multilinear forms motivates us to develop an approximation algorithm for solving problem $(P)$ through tensor relaxation. Therefore, the following spherical constrained multilinear form optimization $(F)$
\begin{equation}
    \begin{aligned}
        (F) \quad & \max \ \re \boldsymbol{F}\left(\boldsymbol{x}^{1}, \boldsymbol{x}^{2}, \ldots, \boldsymbol{x}^{d}\right) \\
         &\operatorname{s.t.} \quad \boldsymbol{x}^{k} \in \mathbf{S}^{n_{k}}, k=1,2, \ldots, d, \notag
    \end{aligned}
\end{equation}
serves as a linkage bridge. 
We first develop a randomized approximation algorithm for problem $(F)$, then utilize it as a subroutine to address the original problem $(P)$.
Note that setting $d=2$ in problem $(F)$ allows it to be reformulated as 
    \begin{equation}\label{eq.d=2}
        \begin{split}
            & \max \ \re\left(\boldsymbol{x}^\top \boldsymbol{A} \boldsymbol{y}\right), \\
         &\operatorname{s.t.} \quad \|\boldsymbol{x}\|^{2} = \|\boldsymbol{y}\|^2 = 1, \boldsymbol{x} \in \mathbb{H}^m, \boldsymbol{y} \in \mathbb{H}^n.
        \end{split}
    \end{equation}
    where $\boldsymbol{A}=\boldsymbol{\mathcal{F}}\in \mathbb{H}^{n_1\times n_2}$ is a commutative quaternion matrix. Recall that $\boldsymbol{A} = A_0 + A_1 \boldsymbol{i} + A_2 \boldsymbol{j} + A_3 \boldsymbol{k}$, and the same decomposition holds for $\boldsymbol{x}$ and $\boldsymbol{y}$. Therefore, problem \eqref{eq.d=2} is equivalent to computing the spectrum norms of the following large-dimension matrix in the real domain, 
    $$
    \begin{array}{ll}
    \underset{x_i, y_i}{\max} & \left(x_0^\top, x_1^\top, x_2^\top, x_3^\top \right)\left(\begin{array}{rrrr}
    A_0 & -A_1 & A_2 &  -A_3 \\
    -A_1 & -A_0 & -A_3 & -A_2 \\
    A_2 & -A_3 & A_0 & -A_1 \\
    -A_3 & -A_2 & -A_1 & -A_0
    \end{array}\right)\left(\begin{array}{l} y_0 \\ y_1 \\ y_2 \\ y_3 \end{array}\right) \\
    \text { s.t. } &\left\|\left(\begin{array}{l} x_0 \\ x_1 \\ x_2 \\ x_3 \end{array}\right)\right\| = \left\|\left(\begin{array}{l} y_0 \\ y_1 \\ y_2 \\ y_3 \end{array}\right)\right\| = 1, \\
    & x_i \in \mathbb{R}^m, \ y_i \in \mathbb{R}^n, \ i = 1, 2, 3, 4,
    \end{array}
    $$
    which can be solved in polynomial time. This yields the following lemma.
    \begin{lemma}\label{le.d=2}
        The problem $(F)$ with $d=2$ can be solved in polynomial time.
    \end{lemma}

We close this subsection by recalling the concept of approximation ratio, which measures the effectiveness of the proposed approximation algorithm.
\begin{definition}
    For any maximization problem $(P)$: $\max_{x\in X}p(x)$ with optimal value $v^*(P)$. A constant $\tau\in(0,1]$ is called the approximation ratio of a polynomial-time approximation algorithm for the problem $(P)$ if the algorithm returns a feasible solution $\hat{x}\in X$ satisfying $p(\hat{x})\ge\tau \cdot v^*(P)$.
\end{definition}

\subsection{A Motivating Example: the Best Rank-One Tensor Approximation}
Quaternion tensors have recently gained popularity in color image processing due to their effective representation ability \cite{miao2020low,chen2020low,chen2019low}. 
One widely used application is the low-rank quaternion approximation model for color images, formulated as:
\begin{equation}\label{eq.low-rank model}
    \min_{\boldsymbol{\mathcal{X}}\in \mH^{m\times n\times t}}\left\{\frac{1}{2}\Vert (\boldsymbol{\mathcal{X}}-\boldsymbol{\mathcal{F}})_{\Omega}\Vert^2: \rank(\boldsymbol{\mathcal{X}})=r\right\},
\end{equation}
where $r \le \min\{m,n,t\}$, and $\Omega$ is the set of observed entries of an $m \times n \times t$ commutative quaternion tensor $\boldsymbol{\mathcal{F}}$. Our focus is on the special case where $r=1$ and $\Omega$ is full observation in model \eqref{eq.low-rank model}. In this scenario, the problem reduces to \textit{the best rank-one approximation} \cite{qi2011best,jiang2015uniqueness,friedland2013best,yang2016rank,yang2016convergence} of a commutative quaternion tensor:
\begin{equation}\label{best rank-one approximation}
    \min_{\boldsymbol{x}^k \in \mH^{n_k},k=1,\cdots,d}\frac{1}{2}\Vert \boldsymbol{x}^1\otimes \cdots \otimes \boldsymbol{x}^d-\boldsymbol{\mathcal{F}}\Vert^2,
\end{equation}
where $\boldsymbol{\mathcal{F}}\in \mH^{n_1\times\cdots\times n_d}$ is a nonzero commutative quaternion tensor. This problem \eqref{best rank-one approximation} can be equivalently modeled by our problem $(P)$ and $(F)$ after some reformulation. For compactness, we defer proofs in the Appendix A. 
\begin{proposition}[Equivalence with $(F)$]\label{prop.equi with F}
     The optimization problem (\ref{best rank-one approximation}) is equivalent to the following problem:
    \begin{equation}\label{tight1}
        \max_{\Vert \bx^k\Vert=1,k=1,\cdots,d} \re\boldsymbol{F}(\bx^1,\cdots,\bx^d). 
    \end{equation}
\end{proposition}
\begin{proposition}[Equivalence with $(P)$]\label{prop.equi with P}
    The optimization problem (\ref{best rank-one approximation}) is equivalent to the following problem:
    \begin{equation}\label{eq.equi with P}
        \max_{\Vert \bx\Vert^2=1,\bx\in\boldsymbol{H}^{\sum_{i=1}^{d}n^k}}\re\boldsymbol{H}(\bx).
    \end{equation}
 \end{proposition}
    Furthermore, combing Lemma \ref{le.d=2} with Proposition \ref{prop.equi with F} gives the following corollary.
    \begin{corollary}
        If $d = 2$ in problem~\eqref{best rank-one approximation}, the problem can be solved in polynomial time.
    \end{corollary}
\section{Probability Inequality in Quaternion Domain}\label{section.inequality}
In this section, we establish a novel probability inequality in the commutative quaternion domain, which is crucial for developing the approximation algorithm to solve the problem $(F)$. We begin by presenting some probability theory over the commutative quaternion domain. For a comprehensive study of general quaternion probability theories, readers can refer to \cite{loots2010development, vakhania2010quaternion, liu2022randomized}.

\begin{definition}\label{definition.distribution}
    A random $n \times 1$ commutative quaternion vector $\boldsymbol{\xi}=\xi_0 + \xi_1\boldsymbol{i} + \xi_2\boldsymbol{j} + \xi_3\boldsymbol{k}$ follows the quaternion normal distribution $\mathcal{QN}(0,4I_n)$ law, if it satisfies $\xi_0,\xi_1,\xi_2,\xi_3 \overset{\text{i.i.d}}{\sim} \mathcal{N}(0,I_n)$. 
\end{definition}

Notice that if $\boldsymbol{\xi} \sim \mathcal{QN}(0, 4I_n)$, then $\|\boldsymbol{\xi}\|^2$ represents the sum of squares of $4n$ independent real variables, each following the $\mathcal{N}(0, 1)$ distribution. Consequently, $\|\boldsymbol{\xi}\|^2$ follows real chi-squared distribution $\chi_{4n}^{2}$ with $4n$ degrees of freedom. Moreover, the concept of uniform distribution on the commutative quaternion sphere follows below, which is a direct extension of the real and complex domain.

\begin{definition}\label{definition.quaternion sphere uniform distribution}
A random quaternion vector $\boldsymbol{\xi}$ is a multivariate uniform distribution on the unit sphere $\mathbf{S}^{n}$, denoted by $\boldsymbol{\xi} \sim \mathbf{S}^{n}$, if $\Vec{\xi} = (\xi_0^\top,\xi_1^\top,\xi_2^\top,\xi_3^\top)^\top \in \mathbb{R}^{4n}$ is a uniform distribution on the real sphere $\mathbf{S}_{\mR}^{4n}$, i.e. $\Vec{\xi} \sim \mathbf{S}_{\mR}^{4n}$.
\end{definition}

The following technical lemma characterizes the property of multivariate uniform distribution on $\mathbf{S}^{n}$, which will be used in establishing the probability inequality.

\begin{lemma}\label{lemma.quaternion sphere uniform distribution property}
A random commutative quaternion vector $\boldsymbol{\xi} \sim \mathbf{S}^{n}$ is equivalent to $\boldsymbol{\eta}/\|\boldsymbol{\eta}\|$, with $\boldsymbol{\eta} \sim \mathcal{QN}(0, 4I_n)$.
\end{lemma}
\begin{proof}
According to Definition \ref{definition.quaternion sphere uniform distribution} and the property of uniform distribution on the sphere in the real domain, we know that $\Vec{\xi} \sim \mathbf{S}_{\mR}^{4n}$ is equivalent to $\Vec{\eta}/\|\Vec{\eta}\|$, where $\Vec{\eta} = (\eta_{ij})^\top \in \mathbb{R}^{4n}$ with $\eta_{ij}$'s are i.i.d. standard random variables for all $i = 0,1,2,3$ and $j = 1, \cdots, n$. Let $\eta_i = (\eta_{i1}, \cdots, \eta_{in})^\top$, $i = 0,1,2,3$, then it follows
$$\eta_0, \eta_1, \eta_2, \eta_3 \overset{\text{i.i.d}}{\sim} \mathcal{N}(0, I_n),$$
which further implies that $$\boldsymbol{\eta} = \eta_0 + \eta_1 \bi + \eta_2 \bj + \eta_3 \bk \sim \mathcal{QN}(0,4I_n).$$ 
Combining with the fact $\|\Vec{\eta}\| = \|\boldsymbol{\eta}\|$ completes the equivalence between $\boldsymbol{\xi}$ and $\boldsymbol{\eta}/\|\boldsymbol{\eta}\|$, with $\boldsymbol{\eta} \sim \mathcal{QN}(0, 4I_n)$.
\end{proof}

Now we are ready to prove the main result of this section.

\begin{theorem}\label{thm.constant probability inequality}
If $\boldsymbol{\xi}$ and $\boldsymbol{a}$ are both uniform distributions on the commutative quaternion unit sphere $\mathbf{S}^{n}$, then for $\gamma>0$ with $\gamma \ln n<n$, there exists a constant $c(\gamma)>0$ such that
$$
\operatorname{Prob}\left\{\operatorname{Re}\left(\boldsymbol{a}^\top \boldsymbol{\xi}\right) \geq \sqrt{\frac{\gamma \ln n}{n}}\right\} \geq \frac{c(\gamma)}{n^{4.5 \gamma} \sqrt{\ln n}}.
$$
\end{theorem}
\begin{proof}
By the symmetry property of the commutative quaternion sphere, without loss of generality, we assume that $\boldsymbol{a} = (1, 0, \cdots, 0)^\top$ is a given vector in $\mathbf{S}^{n}$. Let $\boldsymbol{\eta} \sim \mathcal{QN}(0, 4I_n)$, then according to Lemma \ref{lemma.quaternion sphere uniform distribution property} we know that $\boldsymbol{\xi} = \boldsymbol{\eta} /\|\boldsymbol{\eta}\|$ and $\text{Re}\left(\boldsymbol{a}^\top \boldsymbol{\xi}\right) = \text{Re}(\boldsymbol{\eta}_1) / \|\boldsymbol{\eta}\|$. Moreover, $\text{Re}(\boldsymbol{\eta}_1) \sim \mathcal{N}(0, 1)$ holds obviously from the definition of $\mathcal{QN}(0, 4I_n)$. 

To verify the inequality in this theorem, we consider the following two aspects.
\begin{enumerate}
    \item [(1)] With the above properties in hand, we have that 
$$
\begin{aligned}
\text{Prob}\left\{\text{Re}(\boldsymbol{\eta_1}) \geq 3 \sqrt{\gamma \ln n}\right\} &=\int_{3 \sqrt{\gamma \ln n}}^{+\infty} \frac{1}{\sqrt{2 \pi}} e^{-x^{2} / 2} d x \\
& \geq \int_{3 \sqrt{\gamma \ln n}}^{4 \sqrt{\gamma \ln n}} \frac{1}{\sqrt{2 \pi}} e^{-x^{2} / 2} d x \\
& \geq \int_{3 \sqrt{\gamma \ln n}}^{4 \sqrt{\gamma \ln n}} \frac{1}{\sqrt{2 \pi}} \frac{x}{4 \sqrt{\gamma \ln n}} e^{-x^{2} / 2} d x \\
&=\frac{1}{\sqrt{32 \pi \gamma \ln n}}\left(\frac{1}{n^{4.5 \gamma}}-\frac{1}{n^{8 \gamma}}\right).
\end{aligned}
$$
holds for $n \geq 2$.
    \item [(2)] As for the term $\|\boldsymbol{\eta}\|^2$, according to Definition \ref{definition.distribution}, $\|\boldsymbol{\eta}\|^2$ follows real chi-squared distribution $\chi_{4 n}^{2}$ with $4 n$ degrees of freedom. Then, we can refer to a result on the $\chi^{2}$-distribution estimation by Laurent and Massart (\cite{laurent2000adaptive} Lemma 1):
For any vector $b=\left(b_{1}, b_{2}, \ldots, b_{n}\right)^{\mathrm{T}}$ with $b_{i} \geq 0(i=1,2, \ldots, n)$, denote $z=\sum_{i=1}^{n} b_{i}\left(\eta_{i}^{2}-1\right)$, then for any $t>0$,
\begin{equation}\label{inequality.Laurent and Massart}
\text{Prob}\left\{z \geq 2\|b\| \sqrt{t}+2\|b\|_{\infty} t\right\} \leq e^{-t}. 
\end{equation}    
That is to say, let $b$ be the all-one vector and $t = \frac{3}{4}n$ leads to
$$
\text{Prob}\left\{\|\boldsymbol{\eta}\|^{2} \geq 2\sqrt{3} n + \frac{11}{2}\right\} \leq e^{-\frac{3}{4}n}.
$$
\end{enumerate}
Now, combining the assumption of $\boldsymbol{a}$ with these two inequalities, we get
$$
\begin{aligned}
\text{Prob}\left\{\text{Re}\left(\boldsymbol{a}^\top \boldsymbol{\xi}\right) \geq \sqrt{\frac{\gamma \ln n}{n}}\right\} 
 = &\text{Prob}\left\{\frac{ \text{Re}(\boldsymbol{\eta}_{1})}{ \|\boldsymbol{\eta}\|} \geq \sqrt{\frac{\gamma \ln n}{n}}\right\} \\
 \geq &\text{Prob}\left\{\text{Re}(\boldsymbol{\eta}_{1}) \geq 3 \sqrt{\gamma \ln n},\|\boldsymbol{\eta}\| \leq 3 \sqrt{n}\right\} \\
 \geq &\text{Prob}\left\{\text{Re}(\boldsymbol{\eta}_{1}) \geq 3 \sqrt{\gamma \ln n}\right\} \\
 & - \text{Prob}\left\{\|\boldsymbol{\eta}\| \geq 3 \sqrt{n}\right\} \\
 \geq &\frac{1}{\sqrt{32 \pi \gamma \ln n}}\left(\frac{1}{n^{4.5 \gamma}}-\frac{1}{n^{8 \gamma}}\right) -  e^{-\frac{3}{4}n}.
\end{aligned}
$$
Therefore, there exists $n(\gamma)>0$, depending only on $\gamma$, such that
$$
\begin{aligned}
    \text{Prob}\left\{\text{Re}\left(\boldsymbol{a}^\top \boldsymbol{\xi}\right) \geq \sqrt{\frac{\gamma \ln n}{n}}\right\} &\geq \frac{1}{\sqrt{32 \pi \gamma \ln n}}\left(\frac{1}{n^{4.5 \gamma}}-\frac{1}{n^{8 \gamma}}\right) -  e^{-\frac{3}{4}n}\\ &\geq \frac{1}{2 n^{4.5 \gamma} \sqrt{32 \gamma \pi \ln n}}, \quad \forall n \geq n(\gamma).
\end{aligned}
$$
On the other hand, $0<\gamma<n / \ln n$ implies that
\begin{align*}
    &\text{Prob}\left\{\text{Re}\left(\boldsymbol{a}^\top \boldsymbol{\xi}\right) \geq \sqrt{\frac{\gamma \ln n}{n}}\right\} \\
    > &\text{Prob}\left\{\text{Re}\left(\boldsymbol{a}^\top \boldsymbol{\xi}\right) \geq 1\right\}\\
    = &\text{Prob}\left\{\text{Re}\left(\boldsymbol{\eta}_1\right) \geq \|\boldsymbol{\eta}\| \right\} \geq 0.
\end{align*}
Thus, the following inequality
$$
t(\gamma) : = \min _{n<n(\gamma), 0<\gamma<n / \ln n, n \in \mathbb{Z}} \text{Prob}\left\{\text{Re}\left(\boldsymbol{a}^\top \boldsymbol{\xi}\right) \geq \sqrt{\frac{\gamma \ln n}{n}}\right\} \cdot n^{4.5 \gamma} \sqrt{\ln n}>0
$$
can be obtained, where $t(\gamma)$ depends only on $\gamma$. Finally, choosing $c(\gamma)=\min \{t(\gamma), 1 /(2 \sqrt{32 \gamma \pi})\}$ completes the proof.
\end{proof}

The above result has two key ingredients: the approximation ratio on the left-hand side and the probability bound on the right-hand side. Compared with its real counterpart (Lemma 2.5 in \cite{he2014probability}), our probability bound appears slightly weaker. This is a consequence of the quaternion domain, as an $n$-dimensional quaternion vector drawn from a uniform spherical distribution is composed of $4n$ real random variables. A naive approach would be to use the real representation of the quaternion vectors with the existing real domain result. While yielding the same probability bound, this approach results in a worse approximation ratio. Furthermore, we show that the $n^{-4.5}$ term in our probability bound can be improved to $n^{-(2+\delta)}$ for any $\delta > 0$. This improved result is presented in the following proposition, and its proof is deferred to Appendix B.
    \begin{proposition}\label{prop.epsilon version}
    If $\boldsymbol{\xi}$ and $\boldsymbol{a}$ are both uniform distributions on quaternion sphere $\mathbf{S}^{n}$, then for $\gamma>0$ with $\gamma \ln n<n$ and $\delta > 0$, there exists a constant $c(\gamma, \delta)>0$, such that
    $$
    \operatorname{Prob}\left\{\operatorname{Re}\left(\boldsymbol{a}^\top \boldsymbol{\xi}\right) \geq \sqrt{\frac{\gamma \ln n}{n}}\|\boldsymbol{a}\|\right\} \geq \frac{c(\gamma, \delta)}{n^{(2 + \delta + \delta^2/2) \gamma} \sqrt{\ln n}}.
    $$
    \end{proposition}

Based on the Theorem \ref{thm.constant probability inequality}, we can easily deduce the following corollary for a more general case.
\begin{corollary}\label{cor all a}
If $\boldsymbol{\xi}$ follows a uniform distribution on quaternion sphere $\mathbf{S}^{n}$, then for any $\boldsymbol{a} \in \mathbb{H}^n$ and $\gamma>0$ with $\gamma \ln n<n$, there exists a constant $c(\gamma)>0$ such that
$$
\operatorname{Prob}\left\{\operatorname{Re}\left(\boldsymbol{a}^\top \boldsymbol{\xi}\right) \geq \sqrt{\frac{\gamma \ln n}{n}}\|\boldsymbol{a}\|\right\} \geq \frac{c(\gamma)}{n^{4.5 \gamma} \sqrt{\ln n}}.
$$
\end{corollary}

\section{Approximation for Multilinear Form Optimization}\label{section.algorithm}
In this section, we present a polynomial-time randomized algorithm for problem $(F)$:
\begin{align*}
        (F) \quad & \max \ \re \boldsymbol{F}\left(\boldsymbol{x}^{1}, \boldsymbol{x}^{2}, \ldots, \boldsymbol{x}^{d}\right) \\
         &\operatorname{s.t.} \quad \boldsymbol{x}^{k} \in \mathbf{S}^{n_{k}}, k=1,2, \ldots, d,
\end{align*}
and then establish the performance ratio using the previously derived probability inequality. Without loss of generality, we assume that $n_1\le n_2\le \cdots\le n_d$ in this section.
\begin{algorithm}
\caption{Randomized Algorithm for $(F)$}
\label{Alg: Randomized Alg 1}
1. Randomly and independently generate $\boldsymbol{\xi}^{k}$ uniformly on $\mathbf{S}^{n_{k}}$ for $k=1,2, \ldots, d-2$; \\
2. Solve the following commutative quaternion multilinear form optimization problem
$$
\begin{array}{ll}
\max & \operatorname{Re} \boldsymbol{F}\left(\boldsymbol{\xi}^{1}, \boldsymbol{\xi}^{2}, \ldots, \boldsymbol{\xi}^{d-2}, \boldsymbol{x}^{d-1}, \boldsymbol{x}^{d}\right) \\
\text { s.t. } & \boldsymbol{x}^{d-1} \in \mathbf{S}^{n_{d-1}}, \boldsymbol{x}^{d} \in \mathbf{S}^{n_{d}},
\end{array}
$$
and get its approximate solution $\left(\boldsymbol{\xi}^{d-1}, \boldsymbol{\xi}^{d}\right)$; 
\\
3. Compute the objective value $\operatorname{Re} \boldsymbol{F}\left(\boldsymbol{\xi}^{1}, \boldsymbol{\xi}^{2}, \ldots, \boldsymbol{\xi}^{d}\right)$; \\
4. Repeat the above procedures independently $\ln \frac{1}{\epsilon} \cdot \frac{1}{c(\gamma)^{d-2}} \prod_{k=1}^{d-2} n_k^{4.5 \gamma} \sqrt{\ln n_k}$ times for any given $\epsilon>0$ and $\gamma \in\left(0, \frac{n_{1}}{\ln n_{1}}\right)$, and choose a solution with the largest objective value.
\end{algorithm}

It is worth noting that the approximate ratio of Algorithm \ref{Alg: Randomized Alg 1} mainly depends on the calculation time of step 2. By Lemma \ref{le.d=2}, the step 2 can be solved in polynomial time.
With this guarantee, we can prove the approximation ratio of Algorithm \ref{Alg: Randomized Alg 1}.
\begin{theorem}\label{tm.approximation theorem}
The randomized algorithm \ref{Alg: Randomized Alg 1} solves $\left(F\right)$ with an approximation ratio of $\gamma^{\frac{d-2}{2}}\left(\prod_{k=1}^{d-2} \sqrt{\frac{\ln n_k}{n_k}}\right)$, i.e., for any given $\epsilon > 0$ and $\gamma \in \left(0, \frac{n_1}{\ln n_1}\right)$, a feasible solution $\left(\boldsymbol{y}^{1}, \boldsymbol{y}^{2}, \ldots, \boldsymbol{y}^{d}\right)$ can be generated in polynomial time with probability at least $1-\epsilon$, such that
$$
\operatorname{Re} \boldsymbol{F}\left(\boldsymbol{y}^{1}, \boldsymbol{y}^{2}, \ldots, \boldsymbol{y}^{d}\right) \geq \gamma^{\frac{d-2}{2}}\left(\prod_{k=1}^{d-2} \sqrt{\frac{\ln n_{k}}{n_{k}}}\right) v^*(F),
$$
where $v^*(F)$ is the optimal value of $\left(F\right)$.
\end{theorem}
\begin{proof}
For the problem of degree $d$, our proof is based on mathematical induction on $t=2,3, \ldots, d$. Suppose $\left(\boldsymbol{\xi}^1, \boldsymbol{\xi}^2, \ldots, \boldsymbol{\xi}^d\right)$ is an approximate solution generated by the first three steps of Algorithm \ref{Alg: Randomized Alg 1}. For any $t=2,3, \ldots, d$, we treat $\left(\boldsymbol{\xi}^1, \boldsymbol{\xi}^2, \ldots, \boldsymbol{\xi}^{d-t}\right)$ as given parameters and define the following problem
    \begin{align*}
       (F_t) \quad  & \max \ \operatorname{Re} \boldsymbol{F}\left(\boldsymbol{\xi}^1, \boldsymbol{\xi}^2, \ldots, \boldsymbol{\xi}^{d-t}, \boldsymbol{x}^{d-t+1}, \boldsymbol{x}^{d-t+2} \ldots, \boldsymbol{x}^d\right), \\
         &\operatorname{s.t.} \quad \boldsymbol{x}^k \in \mathbf{S}^{n_k}, k=d-t+1, d-t+2, \ldots, d,
    \end{align*}
whose optimal value is denoted by $v^*\left(F_t\right)$. By applying the first three steps of Algorithm \ref{Alg: Randomized Alg 1} to problem $\left(F_t\right)$, we get a randomly generated feasible solution $\left(\boldsymbol{\xi}^{d-t+1}, \boldsymbol{\xi}^{d-t}, \ldots, \boldsymbol{\xi}^d\right)$ of $\left(F_t\right)$.
In the remaining, we shall prove that $\left(\boldsymbol{\xi}^{d-t+1}, \boldsymbol{\xi}^{d-t+2}, \ldots, \boldsymbol{\xi}^d\right)$ is a $\gamma^{\frac{t-2}{2}}\left(\prod_{k=d-t+1}^{d-2} \sqrt{\frac{\ln n_k}{n_k}}\right)$-approximate solution of $\left(F_t\right)$ with a nontrivial probability. In other words, for any $t=2,3, \ldots, d$, it holds that
\begin{equation}\label{prob bound}
\begin{aligned}
& \underset{\left(\boldsymbol{\xi}^{d-t+1}, \boldsymbol{\xi}^{d-t+2}, \ldots, \boldsymbol{\xi}^d\right)}{\operatorname{Prob}}\left\{\operatorname{Re} \boldsymbol{F}\left(\boldsymbol{\xi}^1, \boldsymbol{\xi}^2, \ldots, \boldsymbol{\xi}^d\right) \geq \gamma^{\frac{t-2}{2}}\left(\prod_{k=d-t+1}^{d-2} \sqrt{\frac{\ln n_k}{n_k}}\right) v^*\left(F_t\right)\right\} \\
& \geq \frac{c(\gamma)^{t-2}}{\prod_{k=d-t+1}^{d-2} n_k^{4.5 \gamma} \sqrt{\ln n_k}}
\end{aligned}
\end{equation}

For the trivial case $t=2$, which coincides with the Step 2 of Algorithm \ref{Alg: Randomized Alg 1} and it could be solved exactly in polynomial-time. Suppose now \eqref{prob bound} holds for $t-1$. To prove that \eqref{prob bound} holds for $t$, we notice that $\left(\boldsymbol{\xi}^1, \boldsymbol{\xi}^2, \ldots, \boldsymbol{\xi}^{d-t}\right)$ are given fixed parameters. Denote $\left(\boldsymbol{z}^{d-t+1}, \boldsymbol{z}^{d-t+2}, \ldots, \boldsymbol{z}^d\right)$ to be an optimal solution of $\left(F_t\right)$, and define the following two events
$$
E_1 = \bigg\{\boldsymbol{\xi}^{d-t+1} \in \mathbf{S}^{n_{d-t+1}}: \begin{array}{c}
    \operatorname{Re} \boldsymbol{F}\left(\boldsymbol{\xi}^1, \ldots, \boldsymbol{\xi}^{d-t}, \boldsymbol{\xi}^{d-t+1}, \boldsymbol{z}^{d-t+2}, \ldots, \boldsymbol{z}^d\right) \\
    \geq \gamma^{\frac{1}{2}}\left(\sqrt{\frac{\ln n_{d-t+1}}{n_{d-t+1}}}\right) v^*\left(F_t\right)
\end{array}\bigg\}
$$
and
$$
\begin{aligned}
E_2 & =  \Bigg \{\boldsymbol{\xi}^{d-t+1} \in E_1, \boldsymbol{\xi}^{d-t+2} \in \mathbf{S}^{n_{d-t+2}}, \ldots, \boldsymbol{\xi}^d \in \mathbf{S}^{n_{d}}: \operatorname{Re} \boldsymbol{F}\left(\boldsymbol{\xi}^1, \ldots, \boldsymbol{\xi}^{d}\right)\\ 
& \geq \gamma^{\frac{t-3}{2}}\left(\prod_{k=d-t+2}^{d-2} \sqrt{\frac{\ln n_{k}}{n_{k}}}\right) \operatorname{Re} \boldsymbol{F}\left(\boldsymbol{\xi}^1, \ldots, \boldsymbol{\xi}^{d-t}, \boldsymbol{\xi}^{d-t+1}, \boldsymbol{z}^{d-t+2}, \ldots, \boldsymbol{z}^d\right) \Bigg \}.
\end{aligned}
$$
Then, we have
\begin{equation}\label{two lower bound}
\begin{aligned}
& \underset{\left(\boldsymbol{\xi}^{d-t+1}, \boldsymbol{\xi}^{d-t+2}, \ldots, \boldsymbol{\xi}^d\right)}{\operatorname{Prob}}\left\{\operatorname{Re} \boldsymbol{F}\left(\boldsymbol{\xi}^1, \boldsymbol{\xi}^2, \ldots, \boldsymbol{\xi}^d\right) \geq \gamma^{\frac{t-2}{2}}\left(\prod_{k=d-t+1}^{d-2} \sqrt{\frac{\ln n_k}{n_k}}\right) v^*\left(F_t\right)\right\} \\
& \geq  \underset{\left(\boldsymbol{\xi}^{d-t+1}, \boldsymbol{\xi}^{d-t+2}, \ldots, \boldsymbol{\xi}^d\right)}{\operatorname{Prob}}\left\{\left(\boldsymbol{\xi}^{d-t+1}, \boldsymbol{\xi}^{d-t+2}, \ldots, \boldsymbol{\xi}^d\right) \in E_2 \ \Big | \ \boldsymbol{\xi}^{d-t+1} \in E_1\right\}\times\\
&\quad\underset{\boldsymbol{\xi}^{d-t+1}}{\operatorname{Prob}}\left\{\boldsymbol{\xi}^{d-t+1} \in E_1\right\}.
\end{aligned}
\end{equation}
To proceed, we provide a lower bound for \eqref{two lower bound}. Firstly, note that $\big(\boldsymbol{z}^{d-t+2},$\\$ \boldsymbol{z}^{d-t+3}, \ldots, \boldsymbol{z}^d\big)$ is a feasible solution of $\left(F_{t-1}\right)$, from which we derive 
$$\operatorname{Re} \boldsymbol{F}\left(\boldsymbol{\xi}^1, \ldots, \boldsymbol{\xi}^{d-t}, \boldsymbol{\xi}^{d-t+1}, \boldsymbol{z}^{d-t+2}, \ldots, \boldsymbol{z}^d\right) \leq v^*\left(F_{t-1}\right),$$
which further implies
$$
\begin{aligned}
& \underset{\left(\boldsymbol{\xi}^{d-t+1}, \boldsymbol{\xi}^{d-t+2}, \ldots, \boldsymbol{\xi}^d\right)}{\operatorname{Prob}}\left\{\left(\boldsymbol{\xi}^{d-t+1}, \boldsymbol{\xi}^{d-t+2}, \ldots, \boldsymbol{\xi}^d\right) \in E_2 \ \Big | \ \boldsymbol{\xi}^{d-t+1} \in E_1\right\} \\
\geq & \underset{\left(\boldsymbol{\xi}^{d-t+1}, \boldsymbol{\xi}^{d-t+2}, \ldots, \boldsymbol{\xi}^d\right)}{\operatorname{Prob}}\Big\{\begin{array}{c}
    \operatorname{Re} \boldsymbol{F}\left(\boldsymbol{\xi}^1, \boldsymbol{\xi}^2, \ldots, \boldsymbol{\xi}^d\right) \\
    \geq \gamma^{\frac{t-3}{2}}\left(\prod_{k=d-t+2}^{d-2} \sqrt{\frac{\ln n_{k}}{n_{k}}}\right) v^*\left(F_{t-1}\right) \ 
\end{array} \Big | \ \boldsymbol{\xi}^{d-t+1} \in E_1 \Big\} \\
\geq & \frac{c(\gamma)^{t-3}}{\prod_{k=d-t+2}^{d-2} n_k^{4.5 \gamma} \sqrt{\ln n_k}},
\end{aligned}
$$
where the last inequality is due to the induction assumption on $t-1$. Secondly, for the term $\underset{\boldsymbol{\xi}^{d-t+1}} {\operatorname{Prob}}\left\{\boldsymbol{\xi}^{d-t+1} \in E_1\right\}$ in \eqref{two lower bound}, it holds that
$$
\begin{aligned}
& \underset{\boldsymbol{\xi}^{d-t+1}}{\operatorname{Prob}}\left\{\boldsymbol{\xi}^{d-t+1} \in E_1\right\} \\
= & \underset{\boldsymbol{\xi}^{d-t+1}}{\operatorname{Prob}}\left\{\begin{array}{c}
     \operatorname{Re} \boldsymbol{F}\left(\boldsymbol{\xi}^1, \ldots, \boldsymbol{\xi}^{d-t}, \boldsymbol{\xi}^{d-t+1}, \boldsymbol{z}^{d-t+2}, \ldots, \boldsymbol{z}^d\right)\\
     \geq \gamma^{\frac{1}{2}}\left(\sqrt{\frac{\ln n_{d-t+1}}{n_{d-t+1}}}\right) v^*\left(F_t\right) 
\end{array}\right\} \\
\geq & \underset{\boldsymbol{\xi}^{d-t+1}}{\operatorname{Prob}}\left\{\begin{array}{cc}
     \operatorname{Re} \boldsymbol{F}\left(\boldsymbol{\xi}^1, \ldots, \boldsymbol{\xi}^{d-t}, \boldsymbol{\xi}^{d-t+1}, \boldsymbol{z}^{d-t+2}, \ldots, \boldsymbol{z}^d\right)\geq \\
    \gamma^{\frac{1}{2}}\left(\sqrt{\frac{\ln n_{d-t+1}}{n_{d-t+1}}}\right) \left\|\boldsymbol{F}\left(\boldsymbol{\xi}^1, \ldots, \boldsymbol{\xi}^{d-t}, \bullet, \boldsymbol{z}^{d-t+2}, \ldots, \boldsymbol{z}^d\right)\right\|
\end{array}\right\} \\
\geq & \frac{c(\gamma)}{n_{d-t+1}^{4.5\gamma}\sqrt{\ln n_{d-t+1}}},
\end{aligned}
$$
where first inequality is because
\begin{align*}
    v_{\max }\left(F_t\right) = &\operatorname{Re} \boldsymbol{F}\left(\boldsymbol{\xi}^1, \ldots, \boldsymbol{\xi}^{d-t}, \boldsymbol{z}^{d-t+1}, \boldsymbol{z}^{d-t+2}, \ldots, \boldsymbol{z}^d\right) \\
    \leq &\left\|\boldsymbol{F}\left(\boldsymbol{\xi}^1, \ldots, \boldsymbol{\xi}^{d-t}, \bullet, \boldsymbol{z}^{d-t+2}, \ldots, \boldsymbol{z}^d\right)\right\|.
\end{align*}
and the last inequality comes from Corollary \ref{cor all a}. 
With the above two bounds established, the lower bound for the right-hand side of \eqref{two lower bound} is showed as
$$
\begin{aligned}
& \underset{\left(\boldsymbol{\xi}^{d-t+1}, \boldsymbol{\xi}^{d-t+2}, \ldots, \boldsymbol{\xi}^d\right)}{\operatorname{Prob}}\left\{\operatorname{Re} \boldsymbol{F}\left(\boldsymbol{\xi}^1, \boldsymbol{\xi}^2, \ldots, \boldsymbol{\xi}^d\right) \geq \gamma^{\frac{t-2}{2}}\left(\prod_{k=d-t+1}^{d-2} \sqrt{\frac{\ln n_k}{n_k}}\right) v^*\left(F_t\right)\right\} \\
\geq & \frac{c(\gamma)^{t-3}}{\prod_{k=d-t+2}^{d-2} n_k^{4.5 \gamma} \sqrt{\ln n_k}} \cdot \frac{c(\gamma)}{n_{d-t+1}^{4.5\gamma}\sqrt{\ln n_{d-t+1}}} \\ 
= & \frac{c(\gamma)^{t-2}}{\prod_{k=d-t+1}^{d-2} n_k^{4.5 \gamma} \sqrt{\ln n_k}}.
\end{aligned}
$$
To this end, the inequality \eqref{prob bound} is proved via induction on $t$. Given that $\left(F_{t=d}\right)$ is equal to $\left(F\right)$, the first three steps of Algorithm \ref{Alg: Randomized Alg 1} can generate an approximate solution for $\left(F\right)$ with an approximation ratio of $\gamma^{\frac{d-2}{2}}\left(\prod_{k=1}^{d-2} \sqrt{\frac{\ln n_k}{n_k}}\right)$ and with a probability of at least $\frac{c(\gamma)^{d-2}}{\prod_{k=1}^{d-2} n_k^{4.5 \gamma} \sqrt{\ln n_k}} = \theta$. Considering the last step of Algorithm \ref{Alg: Randomized Alg 1}, if we independently draw 
$$
\frac{\ln \frac{1}{\epsilon}}{\theta} = \ln \frac{1}{\epsilon} \cdot \frac{1}{c(\gamma)^{d-2}} \prod_{k=1}^{d-2} n_k^{4.5 \gamma} \sqrt{\ln n_k}
$$
trials and choose a solution with the highest objective value, then the probability of success is at least $1 - (1 - \theta)^{\frac{\ln \frac{1}{\epsilon}}{\theta}} \geq 1 - \epsilon$.
\end{proof}
\begin{remark}
    Our algorithm matches the approximation ratio of its real counterpart (Theorem 4.3 in \cite{he2014probability}), but it requires a greater number of independent samples (Line 4 in Algorithm \ref{Alg: Randomized Alg 1}) because our probability bound is slightly weaker.
\end{remark}

\section{Approximation for Commutative Quaternion Homogeneous Polynomial Optimization }\label{section.homogenous}
This section is concerned with the optimization of homogeneous polynomial $\boldsymbol{H}(\bx)$ in the commutative quaternion domain. For clarity, we restate the problem as follows:
\begin{align*}
    (P) \quad & \max \ \re \boldsymbol{H}(\bx) \\
         &\operatorname{s.t.} \quad \bx \in \mathbf{S}^n.
\end{align*}
Using the tensor relaxation method \cite{he2014probability,jiang2014approximation}, we develop Algorithm \ref{Alg: Randomized Alg 2} to solve problem (P). This algorithm uses Algorithm \ref{Alg: Randomized Alg 1} as a subroutine. The following lemma establishes the connection between problem (P) and problem (F) that justifies this design choice.
\begin{algorithm}
\caption{Randomized Algorithm for $(P)$}
\label{Alg: Randomized Alg 2}
1. Solve the relaxed problem $(F)$ through Algorithm \ref{Alg: Randomized Alg 1} and get the solution $(\hat{\bx}^1,\ldots,\hat{\bx}^d)$; \\
2. Find the $\beta_i\ (i=1,\cdots,d)$ satisfying  (\ref{eq.lower bound});\\
3. Compute $\hat{\bx}=\frac{1}{d}\sum_{k=1}^{d}\beta_k\hat{\bx}^k$;\\
4. Determine the optimal solution:\\
\quad (i) When $d$ is odd, the solution is which has the larger objective value between $-\frac{\hat{\bx}}{\|\hat{\bx}\|}$\\ \quad\quad~ and $\frac{\hat{\bx}}{\|\hat{\bx}\|}$;\\
\quad (ii) When $d$ is even, find $\bx'$ satisfying (\ref{eq.even solution}), $\frac{\bx'}{\|\bx'\|}$ is the solution.
\end{algorithm} 
\begin{lemma}\label{le.linkage}
    Suppose $\bx^1,\bx^2,\cdots,\bx^d\in\mH^n$, and $\xi_1,\xi_2,\cdots,\xi_d$ are i.i.d. symmetric Bernoulli random variables (taking $1$ and $-1$ with equal probability). For any super-symmetric tensor $\boldsymbol{\mathcal{F}}\in\mH^{n^d}$ with its associated multilinear form $\boldsymbol{F}$ and homogeneous polynomial $\boldsymbol{H}$, it holds that 
    $$
    E\left[\mathop{\Pi}
    \limits_{i=1}^{d}\xi_i\boldsymbol{H}\left(\sum_{k=1}^{d}\xi_k\bx^k\right)\right]=d!\boldsymbol{F}(\bx^1,\bx^2,\cdots,\bx^d).
    $$
\end{lemma}
\begin{proof}
    Observe that 
    \begin{align*}
        &E\left[\mathop{\Pi}
    \limits_{i=1}^{d}\xi_i\boldsymbol{H}\left(\sum_{k=1}^{d}\xi_k\bx^k\right)\right]\\
    =&E\left[\mathop{\Pi}
    \limits_{i=1}^{d}\xi_i\sum_{1\le k_1, k_2,\cdots,k_d\le d}\boldsymbol{F}\left(\xi_{k_1}\bx^{k_1},\xi_{k_2}\bx^{k_2},\cdots,\xi_{k_d}\bx^{k_d}\right)\right]\\
    =&\sum_{1\le k_1, k_2,\cdots,k_d\le d}E\left[\left(\mathop{\Pi}
    \limits_{i=1}^{d}\xi_i\right)\left(\mathop{\Pi}
    \limits_{j=1}^{d}\xi_{k_j}\right)\boldsymbol{F}\left(\bx^{k_1},\bx^{k_2},\cdots,\bx^{k_d}\right)\right].
    \end{align*}
    Now, to simplify the right-hand formula, let us break it down into two cases. Specifically, if $(k_1,k_2,\cdots,k_d)\in\Pi(1,2,\cdots,d)$, i.e., a permutation of $\{1,2,\cdots,d\}$, then
    $$
    E\left[\left(\mathop{\Pi}
    \limits_{i=1}^{d}\xi_i\right)\left(\mathop{\Pi}
    \limits_{j=1}^{d}\xi_{k_j}\right)\right]=E\left[\mathop{\Pi}
    \limits_{i=1}^{d}\xi_i^2\right]=1;
    $$
    otherwise, there exists $k_0\in\left[1,d\right]$ and $k_0\neq k_j$ for all $j=1,2,\cdots,d$ such that
    $$
    E\left[\left(\mathop{\Pi}
    \limits_{i=1}^{d}\xi_i\right)\left(\mathop{\Pi}
    \limits_{j=1}^{d}\xi_{k_j}\right)\right]=E\left[\xi_{k_0}\right]E\left[\left(\mathop{\Pi}
    \limits_{1\le i\le d,i\neq k_0}^{d}\xi_i\right)\left(\mathop{\Pi}
    \limits_{j=1}^{d}\xi_{k_j}\right)\right]=0.
    $$
    Since the number of different permutations of $\{1,2,\cdots,d\}$ is $d!$, the claimed relation holds by taking into account the super-symmetric property of $\boldsymbol{F}$.
\end{proof}


We are now ready to prove the following main theorem of the paper.

\begin{theorem}\label{Thm:poly-approx}
    Let $\tau(P)=d^{-d}d!\left(\gamma\cdot\frac{\ln n}{n}\right)^{\frac{d-2}{2}}$, denote $v^*(P)$ and $\underline{v}(P)$ as the optimal values for problems (P) and $min_{\mathbf{x}\in \boldsymbol{B}^n} \re \bH(\bx)$, respectively. Then we have:
    \begin{enumerate}
        \item [(1)] If d is odd, then the randomized Algorithm \ref{Alg: Randomized Alg 2} solves $\left(P\right)$ with an approximation ratio of $d^{-d}d!\left(\gamma\cdot\frac{\ln n}{n}\right)^{\frac{d-2}{2}}$. That is, for any given $\epsilon > 0$ and $\gamma \in \left(0, \frac{n}{\ln n}\right)$, a feasible solution $\bx\in\mH^n$ can be generated in polynomial time with a probability at least $1-\epsilon$, such that
        $$
        \re \boldsymbol{H}(\bx)\ge \tau(P)v^*(P).
        $$
        \item [(2)] If $d$ is even, then for any given $\epsilon > 0$ and $\gamma \in \left(0, \frac{n}{\ln n}\right)$, the randomized Algorithm \ref{Alg: Randomized Alg 2} finds a feasible solution $\bx\in\mH^n$ for problem $(P)$ in polynomial time with a probability at least $1-\epsilon$, such that
        $$\re\boldsymbol{H}\left(\bx\right)-\underline{v}(P)\ge2\tau(P)v^*(P)
        $$
    \end{enumerate}
\end{theorem}
\begin{proof}
        By relaxing problem $(P)$ to problem $(F)$, Algorithm \ref{Alg: Randomized Alg 1} can be used to find a solution $(\hat{\bx}^1,\hat{\bx}^2,\cdots,\hat{\bx}^d)$ with an approximation ratio of $\left(\gamma\cdot\frac{\ln n}{n}\right)^{\frac{d-2}{2}}$, i.e., 
        \begin{equation}\label{eq.ratio F}
            \re \boldsymbol{F}(\hat{\bx}^1,\hat{\bx}^2,\cdots,\hat{\bx}^d)\ge \left(\gamma\cdot\frac{\ln n}{n}\right)^{\frac{d-2}{2}}v^*(F).
        \end{equation}
        To proceed, we discuss the approximation ratio in two cases: when $d$ is odd and when $d$ is even.
        \begin{enumerate}
        \item [(1)]
          When $d$ is odd, it is obvious that $\re\boldsymbol{H}(-\bx)=-\re\boldsymbol{H}(\bx)$. From Lemma \ref{le.linkage}, we know there must exist values $\beta_i = \pm 1 (i=1,\cdots,d)$ such that 
        \begin{equation}\label{eq.lower bound}
            \begin{split}
                \re\mathop{\Pi}\limits_{i=1}^{d}\beta_i\boldsymbol{H}\left(\frac{1}{d}\sum_{k=1}^{d}\beta_k\hat{\bx}^k\right)&\ge\re d^{-d}d!\boldsymbol{F}(\hat{\bx}^1,\hat{\bx}^2,\cdots,\hat{\bx}^d)\\
                &\ge d^{-d}d!\left(\gamma\cdot\frac{\ln n}{n}\right)^{\frac{d-2}{2}}v^*(F)\\
                &= \tau(P)v^*(F) \\
                &\ge\tau(P)v^*(P),
            \end{split}
        \end{equation}
        where the second inequality is due to (\ref{eq.ratio F}). Denote the random vector $\hat{\bx}=\frac{1}{d}\sum_{k=1}^{d}\beta_k\hat{\bx}^k$, it is easy to see that
        \begin{equation}\label{eq.hat_x bound}
            \Vert \hat{\bx}\Vert=\Vert \frac{1}{d}\sum_{k=1}^{d}\beta_k\hat{\bx}^k\Vert\le\frac{1}{d}\sum_{k=1}^d\Vert\beta_k\hat{\bx}^k\Vert = 1.
        \end{equation}
        Therefore, we have
            $$
            \begin{aligned}
                \max\left\{\re\boldsymbol{H}\left(\frac{-\hat{\bx}}{\Vert \hat{\bx}\Vert}\right),\re\boldsymbol{H}\left(\frac{\hat{\bx}}{\Vert \hat{\bx}\Vert}\right)\right\}& = \Vert \hat{\bx}\Vert^{-d}\left| \re\boldsymbol{H}\left(\hat{\bx}\right)\right|\\
                &\ge\re\mathop{\Pi}\limits_{i=1}^{d}\beta_i\boldsymbol{H}\left(\frac{1}{d}\sum_{k=1}^{d}\beta_k\hat{\bx}^k\right)\\
                &\ge\tau(P)v^*(P)
            \end{aligned}
            $$
            where the second and last inequalities are derived from \eqref{eq.hat_x bound} and \eqref{eq.lower bound}, respectively.
            \item [(2)] When $d$ is even, let $\hat{\bx}_{\xi}= \sum_{k=1}^{d}\xi_k\hat{\bx}^k$ where $\xi_i$ satisfies the condition we presented in Lemma 5.1. From Lemma \ref{le.linkage}, we have
            \begin{equation}
                \begin{split}
                    &d^{-d} d!\re \boldsymbol{F}(\hat{\bx}^1,\hat{\bx}^2,\cdots,\hat{\bx}^d)\\
                    =& E\left[\mathop{\Pi}
                    \limits_{i=1}^{d}\xi_i\re\boldsymbol{H}\left(\frac{1}{d}\sum_{k=1}^{d}\xi_k\hat{\bx}^k\right)\right]\\
                    =&E\left[\mathop{\Pi}
                    \limits_{i=1}^{d}\xi_i\re\boldsymbol{H}\left(\frac{1}{d}\hat{\bx}_{\xi}\right)\right]\\
                    =&E\left[\mathop{\Pi}
                    \limits_{i=1}^{d}\xi_i\left(\re\boldsymbol{H}\left(\frac{1}{d}\hat{\bx}_{\xi}\right)-\underline{v}(P)\right)\right],
                \end{split}\nonumber
            \end{equation}
            where the last equality holds because of the identity $E\left(\mathop{\Pi}
            \limits_{i=1}^{d}\xi_i\right)=0$. Applying Tower's rule implies
            \begin{equation}
                \begin{split}
                    &E\left[\mathop{\Pi}
                    \limits_{i=1}^{d}\xi_i\left(\re\boldsymbol{H}\left(\frac{1}{d}\hat{\bx}_{\xi}\right)-\underline{v}(P)\right)\right]\\
                    =&E\left[\mathop{\Pi}
                    \limits_{i=1}^{d}\xi_i\left(\re\boldsymbol{H}\left(\frac{1}{d}\hat{\bx}_{\xi}\right)-\underline{v}(P)\right)|\mathop{\Pi}
                    \limits_{i=1}^{d}\xi_i=1\right]P\left(\mathop{\Pi}
                    \limits_{i=1}^{d}\xi_i=1\right)\\
                    &+E\left[\mathop{\Pi}
                    \limits_{i=1}^{d}\xi_i\left(\re\boldsymbol{H}\left(\frac{1}{d}\hat{\bx}_{\xi}\right)-\underline{v}(P)\right)|\mathop{\Pi}
                    \limits_{i=1}^{d}\xi_i=-1\right]P\left(\mathop{\Pi}
                    \limits_{i=1}^{d}\xi_i=-1\right).
                \end{split}\nonumber
            \end{equation}
            Note that $\hat{\bx}_{\xi}/d \in \boldsymbol{B}^n$, which implies that the term $\mathop{\Pi}\limits_{i=1}^{d} \xi_i\left(\re\boldsymbol{H}\left(\hat{\bx}_{\xi}/d\right)-\underline{v}(P)\right)$ is nonpositive under the condition $ \mathop{\Pi}\limits_{i=1}^{d}\xi_i=-1$, and thus it can be dropped. Consequently, we have
            \begin{equation}
                \begin{split}
                    &d^{-d} d!\re \boldsymbol{F}(\hat{\bx}^1,\hat{\bx}^2,\cdots,\hat{\bx}^d)\\
                    =&E\left[\mathop{\Pi}
                    \limits_{i=1}^{d}\xi_i\left(\re\boldsymbol{H}\left(\frac{1}{d}\hat{\bx}_{\xi}\right)-\underline{v}(P)\right)\right]\\
                    \le &E\left[\mathop{\Pi}
                    \limits_{i=1}^{d}\xi_i\left(\re\boldsymbol{H}\left(\frac{1}{d}\hat{\bx}_{\xi}\right)-\underline{v}(P)\right)|\mathop{\Pi}
                    \limits_{i=1}^{d}\xi_i=1\right]P\left(\mathop{\Pi}
                    \limits_{i=1}^{d}\xi_i=1\right)\\
                    =&\frac{1}{2}E\left[\re\boldsymbol{H}\left(\frac{1}{d}\hat{\bx}_{\xi}\right)-\underline{v}(P)|\mathop{\Pi}
                    \limits_{i=1}^{d}\xi_i= 1\right].
                \end{split}\nonumber
            \end{equation}
            Therefore, we know there must exist values $\beta_i = \pm 1 (i=1,\cdots,d)$, $\mathop{\Pi}
                    \limits_{i=1}^{d} \beta_i = 1$ such that $\bx' = \sum_{k=1}^d \beta_k \hat \bx^k$ which satisfies
            \begin{equation}\label{eq.even solution}
                \begin{split}
                    \re\boldsymbol{H}\left(\frac{1}{d}\bx'\right)-\underline{v}(P)
                    &\ge E\left[\re\boldsymbol{H}\left(\hat{\bx}_{\xi}\right)-\underline{v}(P) |\mathop{\Pi}
                    \limits_{i=1}^{d}\xi_i= 1\right]\\
                    &\ge 2 d^{-d} d!\re \boldsymbol{F}(\hat{\bx}^1,\hat{\bx}^2,\cdots,\hat{\bx}^d)\\
                    &\ge2 d^{-d} d!\left(\gamma\cdot\frac{\ln n}{n}\right)^{\frac{d-2}{2}}v^*(P).
                \end{split}
            \end{equation}
            Since $\|\bx'\| \le d$, combining the above inequalities follows 
            $$
            \begin{aligned}
                \re\boldsymbol{H}\left(\frac{\bx'}{\Vert \bx'\Vert}\right)-\underline{v}(P)
            &\ge2d!d^{-d}\left(\gamma\cdot\frac{\ln n}{n}\right)^{\frac{d-2}{2}}v^*(P)\\
            &\ge 2\tau(P)v^*(P).
            \end{aligned}
            $$
        \end{enumerate}
        where the second inequality is obtained from the definition of $\tau(P)$, and the proof is completed.
\end{proof} 

    \begin{remark}
        Similar to Algorithm \ref{Alg: Randomized Alg 1}, Algorithm \ref{Alg: Randomized Alg 2} attains the same approximation ratio as its real counterpart \cite{he2014probability}. However, a large number of independent samples is required because it uses Algorithm \ref{Alg: Randomized Alg 1} as a subroutine.
    \end{remark}

\section{Numerical experiments}\label{sec.numerical}
In this section, we will verify the rationality of our theoretical analysis through numerical experiments. We construct a special problem and give an upper bound for obtaining the approximate ratio . All the numerical computations are done on a MacBook Pro 13-inch (2022) with Apple M2 and 8 GB of RAM. The supporting software is MATLAB R2023b.

To demonstrate the performance of Algorithm \ref{Alg: Randomized Alg 1}, we test it on problem (F) with fixed parameters ($d=3$ and $n_1=n_2=n_3=n$). However, because the problem is NP-hard, the true optimal value cannot be computed for comparison. Therefore, we construct a special instance of the problem where an upper bound for the optimal value can be derived. This upper bound then serves as a reference for evaluating our algorithm's performance. We summarize it as the following proposition and put the detailed proof in the Appendix B.
\begin{proposition}\label{theorem.upper bound}
    Suppose $d = 3$ and denote $\mathbf{\mathcal{F}}=\mathcal{F}_0 + \mathcal{F}_1\textbf{i} +\mathcal{F}_2\textbf{j} + \mathcal{F}_3\textbf{k}\in \mathbb{H}^{n_1\times n_2\times n_3}$. If we let $\mathcal{F}_0 = \textbf{1}_{n_1\times n_2\times n_3}$ and $\mathcal{F}_1 = \mathcal{F}_2 = \mathcal{F}_3 = \textbf{0}_{n_1\times n_2\times n_3}$, then $v_{\text{upper}} = 2\sqrt{n_1n_2n_3}$ is an upper bound for problem (F), i.e. 
    $$
        v^*(F)\le v_{upper},
    $$
    where $\textbf{1}_{n_1\times n_2\times n_3},\textbf{0}_{n_1\times n_2\times n_3}\in \mathbb{R}^{n_1\times n_2\times n_3}$ are all-ones tensors and all-zeros tensors respectively.
\end{proposition}
\begin{remark}
    In the real domain, special instances of problem (F) with known optimal values can be constructed using the property $\|\mathrm{vec}(x\otimes y)\| = \|x\|\,\|y\|$ for real vectors $x \in \mathbb{R}^l$ and $y \in \mathbb{R}^m$ (see Section 3.5 of \cite{hu2025spectral}). The notation $\mathrm{vec}(\cdot)$ denotes vectorization, which stacks the entries of a matrix into a vector row by row. However, this property does not hold for commutative quaternion vectors. This limitation restricts us to constructing a special instance for which an upper bound is known, rather than the exact optimal value.
\end{remark}
For fixed number of trials, We report the average and worst-case ratio relative to the upper bound over 20 runs for $n=2, 3, 4, 5, 6, 7$. The corresponding numerical results are shown in Table \ref{tab:approximation_ratios_2_3}, Table \ref{tab:approximation_ratios_4_5}, Table \ref{tab:approximation_ratios_6_7}.
\begin{table}[ht]
    \centering
    \caption{Approximation ratios (i.e., objective value / theoretical upper bound) of Algorithm~\ref{Alg: Randomized Alg 1} for $n=2$ and $n=3$ over 20 runs. The average and worst-case ratios are reported for various iteration numbers.}
    \label{tab:approximation_ratios_2_3}
    \begin{tabular}{c ccc c}
        \toprule
         & \multicolumn{2}{c}{n=2} & \multicolumn{2}{c}{n=3}  \\
        \cmidrule(lr){2-3} \cmidrule(lr){4-5} 
         Number of trials & Average ratio & Worst ratio & Average ratio & Worst ratio \\
        \midrule
        1 & 0.4124 & 0.1918 & 0.3197 & 0.1660 \\
        5 & 0.5494 & 0.4278 & 0.4348 & 0.3740\\
        10 & 0.5729 & 0.4927 & 0.5244 & 0.4323 \\
        20 & 0.6042 & 0.5357 & 0.5255 & 0.4424  \\
        50 & 0.6274 & 0.5889 & 0.5590 & 0.4971 \\
        100 & 0.6547 & 0.6060 & 0.5952  & 0.5303 \\
        500 & 0.6737 & 0.6530 & 0.6148  & 0.5690 \\
        1000 & 0.6791 & 0.6615 & 0.6350  & 0.6075 \\
        10000 & 0.6941 & 0.6867 & 0.6620 & 0.6437\\
        \bottomrule
    \end{tabular}
\end{table}
\begin{table}[ht]
    \centering
    \caption{Approximation ratios (i.e., objective value / theoretical upper bound) of Algorithm~\ref{Alg: Randomized Alg 1} for $n=4$ and $n=5$ over 20 runs. The average and worst-case ratios are reported for various iteration numbers.}
    \label{tab:approximation_ratios_4_5}
    \begin{tabular}{c ccc c}
        \toprule
         & \multicolumn{2}{c}{n=4} & \multicolumn{2}{c}{n=5}  \\
        \cmidrule(lr){2-3} \cmidrule(lr){4-5} 
         Number of trials & Average ratio & Worst ratio & Average ratio & Worst ratio \\
        \midrule
        1 & 0.2521 & 0.1254 & 0.2748 & 0.1585 \\
        5 & 0.4275 & 0.3005 & 0.3694 & 0.2570\\
        10 & 0.4382 & 0.3790 & 0.3973 & 0.3214 \\
        20 & 0.4595 & 0.3828 & 0.4365 & 0.3721  \\
        50 & 0.5190 & 0.4547 & 0.4787 & 0.3826 \\
        100 & 0.5291 & 0.4568 & 0.4778  & 0.4326 \\
        500 & 0.5625 & 0.5293 & 0.5351  & 0.4751 \\
        1000 & 0.5843 & 0.5389 & 0.5474  & 0.4992 \\
        10000 & 0.6195 & 0.5856 & 0.5856 & 0.5477\\
        \bottomrule
    \end{tabular}
\end{table}
\begin{table}[ht]
    \centering
    \caption{Approximation ratios (i.e., objective value / theoretical upper bound) of Algorithm~\ref{Alg: Randomized Alg 1} for $n=6$ and $n=7$ over 20 runs. The average and worst-case ratios are reported for various iteration numbers.}
    \label{tab:approximation_ratios_6_7}
    \begin{tabular}{c ccc c}
        \toprule
         & \multicolumn{2}{c}{n=6} & \multicolumn{2}{c}{n=7}  \\
        \cmidrule(lr){2-3} \cmidrule(lr){4-5} 
         Number of trials & Average ratio & Worst ratio & Average ratio & Worst ratio \\
        \midrule
        1 & 0.2216 & 0.1368 & 0.2004 & 0.1240 \\
        5 & 0.3484 & 0.2032 & 0.3007 & 0.2222\\
        10 & 0.3608 & 0.2614 & 0.3357 & 0.2651 \\
        20 & 0.4041 & 0.3284 & 0.3798 & 0.2973  \\
        50 & 0.4293 & 0.3470 & 0.3998 & 0.3445 \\
        100 & 0.4494 & 0.3821 & 0.4296  & 0.3689 \\
        500 & 0.4914 & 0.4573 & 0.4684  & 0.4145 \\
        1000 & 0.5077 & 0.4799 & 0.4765  & 0.4370 \\
        10000 & 0.5585 & 0.5233 & 0.5252 & 0.4982\\
        \bottomrule
    \end{tabular}
\end{table}
\section{Conclusions}
In this paper, we study commutative quaternion polynomial optimization with spherical constraint, which includes the best rank-one tensor approximation over commutative quaternion domain as a special case. To the best of our knowledge, this is the first attempt to study this kind of problem. The probability inequality with uniform commutative quaternion random variables over unit sphere is established. Based on this probability inequality, we propose polynomial-time randomized algorithms for the homogeneous polynomial problem and its multilinear form relaxation, and prove that they have the worst case approximation ratio guarantee.



\appendix
\section*{Appendix A: Missing Proofs in Section \ref{sec.preliminary}}\label{app.A}
We present proofs of proposition \ref{prop.equi with F} and proposition \ref{prop.equi with P} which clarifies the optimization problem (\ref{best rank-one approximation}) is equivalent to model $(F)$ and model $(P)$. For the convenience of readers, we restate these propositions here.
\setcounter{proposition}{0}
\renewcommand\theproposition{A.1}
\begin{proposition}[Equivalence with $(F)$]
     The optimization problem (\ref{best rank-one approximation}) is equivalent to the following problem:
    \begin{equation}\label{tight1}
        \max_{\Vert \bx^k\Vert=1,k=1,\cdots,d} \re\boldsymbol{F}(\bx^1,\cdots,\bx^d). 
    \end{equation}
\end{proposition}
\begin{proof}
It is easy to see the problem \eqref{best rank-one approximation} is equivalent to 
$$
\min_{\lambda \in \mR, \Vert \boldsymbol{x}^k\Vert=1,k=1,\cdots,d}\frac{1}{2}\Vert \lambda\boldsymbol{x}^1\otimes \cdots \otimes \boldsymbol{x}^d-\boldsymbol{\mathcal{F}}\Vert^2,
$$
and the optimal $\lambda$ satisfies
\begin{equation}
    \begin{split}
        \min_{\lambda\in\mR}\Vert \lambda\boldsymbol{x}^1\otimes \cdots \otimes \boldsymbol{x}^d-\boldsymbol{\mathcal{F}}\Vert^2
        &=\min_{\lambda\in \mR}\left( \Vert \boldsymbol{\mathcal{F}}\Vert^2-2\lambda\re\boldsymbol{F}(\bx^1,\cdots,\bx^d)+\lambda^2\right)\\
        &=\Vert \boldsymbol{\mathcal{F}}\Vert^2-(\re\boldsymbol{F}(\bx^1,\cdots,\bx^d))^2. \notag
    \end{split}
\end{equation}
 With multilinearity, we can solve the following problem 
 \begin{equation}
      \max_{\Vert \bx^k\Vert=1,k=1,\cdots,d}\vert \re\boldsymbol{F}(\bx^1,\cdots,\bx^d)\vert = \max_{\Vert \bx^k\Vert=1,k=1,\cdots,d} \re\boldsymbol{F}(\bx^1,\cdots,\bx^d)\notag
 \end{equation} 
to get the solution of \eqref{best rank-one approximation}.
\end{proof}

To demonstrate the equivalence of model $(P)$, we first equate the problem \eqref{best rank-one approximation} 
\begin{equation*}
    \min_{\boldsymbol{x}^k \in \mH^{n_k},k=1,\cdots,d}\frac{1}{2}\Vert \boldsymbol{x}^1\otimes \cdots \otimes \boldsymbol{x}^d-\boldsymbol{\mathcal{F}}\Vert^2
\end{equation*}
to the following intermediary problem:
      \begin{equation}\label{tight2}
     \max_{\sum_{k=1}^{d}\Vert \bx^k\Vert^2=d}\re\boldsymbol{F}(\bx^1,\cdots,\bx^d).
 \end{equation}
 and then propose the equivalence between optimization problem (\ref{best rank-one approximation}) and model $(P)$.
\setcounter{proposition}{1}
\renewcommand\theproposition{A.2}
\begin{proposition}[Equivalence with $(P)$]
    The optimization problem (\ref{best rank-one approximation}) is equivalent to the following problem:
    \begin{equation}\label{eq.equi with P}
        \max_{\Vert \bx\Vert^2=1,\bx\in\boldsymbol{H}^{\sum_{i=1}^{d}n^k}}\re\boldsymbol{H}(\bx).
    \end{equation}
 \end{proposition}
 \begin{proof}
    First of all, we claim that if $(\bx^1_*,\cdots,\bx^d_*)$ is the optimal solution of the intermediary problem (\ref{tight2})
    , then $\left( \frac{\bx^1_*}{\Vert \bx^1_*\Vert},\cdots,\frac{\bx^d_*}{\Vert \bx^d_*\Vert}\right)$ serves as the solution of (\ref{tight1}). Consequently, solving (\ref{best rank-one approximation}) is equivalent to solving (\ref{tight2}). 
    Suppose $(\bx^1_*,\cdots,\bx^d_* )$ is the optimal solution of (\ref{tight2}), then we have 
     $$
     \re \boldsymbol{F}\left( \frac{\bx^1_*}{\Vert \bx^1_*\Vert},\cdots,\frac{\bx^d_*}{\Vert \bx^d_*\Vert}\right)=\re \frac{\boldsymbol{F}(\bx^1_*,\cdots,\bx^d_*)}{\Pi_{i=1}^{d}\Vert \bx^i_*\Vert}\ge\re\boldsymbol{F}(\bx^1_*,\cdots,\bx^d_*),
     $$
     the above inequality is due to
     $$
     \left(\Pi_{i=1}^{d}\Vert \bx^i_*\Vert^2\right)^{1/d}\le \frac{1}{d}\sum_{i=1}^{d}\Vert \bx^i_*\Vert^2=1.
     $$
     Finally, since (\ref{tight2}) is a relaxation of (\ref{tight1}), the claim holds.
     
     With the above claim, we can give the equivalence between optimization problem (\ref{best rank-one approximation}) and model $(P)$. Let $\bx=((\bx^1)^\top,\cdots,(\bx^d)^\top)^\top\in\boldsymbol{H}^{\sum_{i=1}^{d}n^k}$. Since $\boldsymbol{F}(\bx^1,\cdots,\bx^d)$ is a $d$-degree homogeneous polynomial, we can find a $\sum_{i=1}^{d}n^k$-dimensional $d$-th degree homogeneous polynomial function $\boldsymbol{H}(\bx)$ such that
$$
\boldsymbol{H}(\bx)=\boldsymbol{F}(\bx^1,\cdots,\bx^d).
$$
Thus, we have 
$$
\max_{\sum_{k=1}^{d}\Vert \bx^k\Vert^2=d}\re\boldsymbol{F}(\bx^1,\cdots,\bx^d)=\max_{\Vert \bx\Vert^2=d}\re\boldsymbol{H}(\bx)=\max_{\Vert \bx\Vert^2=1}\re\boldsymbol{H}(\sqrt{d}\bx)
$$
which can be solved by the homogeneous polynomial model (\ref{eq.equi with P}).
\end{proof}

\section*{Appendix B: Proof of Proposition \ref{prop.epsilon version}}\label{app.B}
\textbf{Proof of Proposition \ref{prop.epsilon version}:}

By the similar argument in \ref{thm.constant probability inequality}, we obtain that $\text{Re}(\boldsymbol{\eta}_1) \sim \mathcal{N}(0, 1)$ and $\|\boldsymbol{\eta}\|^2$ follows real chi-squared distribution $\chi_{4 n}^{2}$. For the term $\text{Re}(\boldsymbol{\eta}_1)$, the lower limit of the integral can be relaxed with respect to $\delta$,
$$
\begin{aligned}
\text{Prob}\left\{\text{Re}(\boldsymbol{\eta_1}) \geq 2 \sqrt{\gamma \ln n}\right\} &=\int_{2 \sqrt{\gamma \ln n}}^{+\infty} \frac{1}{\sqrt{2 \pi}} e^{-x^{2} / 2} d x \\
& \geq \int_{\left(2+\delta\right) \sqrt{\gamma \ln n}}^{4 \sqrt{\gamma \ln n}} \frac{1}{\sqrt{2 \pi}} e^{-x^{2} / 2} d x \\
& \geq \int_{\left(2+\delta\right) \sqrt{\gamma \ln n}}^{4 \sqrt{\gamma \ln n}} \frac{1}{\sqrt{2 \pi}} \frac{x}{4 \sqrt{\gamma \ln n}} e^{-x^{2} / 2} d x \\
&=\frac{1}{\sqrt{32 \pi \gamma \ln n}}\left(\frac{1}{n^{(2 + \delta + \delta^2/2) \gamma}}-\frac{1}{n^{8 \gamma}}\right).
\end{aligned}
$$
And for the term $\|\boldsymbol{\eta}\|^2$, let $b$ to be the all-one vector and $t = \frac{\delta^2}{2}n$ in the inequality \eqref{inequality.Laurent and Massart} leads to
$$
\text{Prob}\left\{\|\boldsymbol{\eta}\|^{2} \geq 2\sqrt{2}\delta n + \delta^2 n + 4n\right\} \leq e^{-\frac{\delta^2}{2}n}.
$$
Combine with these two inequalities, we obtain that
$$
\begin{aligned}
\text{Prob}\left\{\text{Re}\left(\boldsymbol{a}^\top \boldsymbol{\xi}\right) \geq \sqrt{\frac{\gamma \ln n}{n}}\right\} 
& = \text{Prob}\left\{\frac{ \text{Re}(\boldsymbol{\eta}_{1})}{ \|\boldsymbol{\eta}\|} \geq \sqrt{\frac{\gamma \ln n}{n}}\right\} \\
& \geq \text{Prob}\left\{\text{Re}(\boldsymbol{\eta}_{1}) \geq \left(2+\delta\right) \sqrt{\gamma \ln n},\|\boldsymbol{\eta}\| \leq \left(2+\delta\right) \sqrt{n}\right\} \\
& \geq \text{Prob}\left\{\text{Re}(\boldsymbol{\eta}_{1}) \geq \left(2+\delta\right) \sqrt{\gamma \ln n}\right\} - \text{Prob}\left\{\|\boldsymbol{\eta}\| \geq \left(2+\delta\right) \sqrt{n}\right\} \\
& \geq \frac{1}{\sqrt{32 \pi \gamma \ln n}}\left(\frac{1}{n^{(2 + \delta + \delta^2/2) \gamma}}-\frac{1}{n^{8 \gamma}}\right) -  e^{-\frac{\delta^2}{2}n}.
\end{aligned}
$$
Therefore, there exists $n(\gamma, \delta)>0$, depending on $\gamma$ and $\delta$, such that
$$
\begin{aligned}
    \text{Prob}\left\{\text{Re}\left(\boldsymbol{a}^\top \boldsymbol{\xi}\right) \geq \sqrt{\frac{\gamma \ln n}{n}}\right\} &\geq \frac{1}{\sqrt{32 \pi \gamma \ln n}}\left(\frac{1}{n^{(2 + \delta + \delta^2/2) \gamma}}-\frac{1}{n^{8 \gamma}}\right) -  e^{-\frac{\delta^2}{2}n}\\
    &\geq \frac{1}{2 n^{(2 + \delta + \delta^2/2) \gamma} \sqrt{32 \pi \gamma \ln n}} \quad \forall n \geq n(\gamma, \delta).
\end{aligned}
$$
On the other hand, $0<\gamma<n / \ln n$ implies that $\text{Prob}\left\{\text{Re}\left(\boldsymbol{a}^\top \boldsymbol{\xi}\right) \geq \sqrt{\frac{\gamma \ln n}{n}}\right\}>0$. Therefore,
$$
\min _{n<n(\gamma, \delta), \gamma \ln n<n, n \in \mathbb{Z}} \text{Prob}\left\{\text{Re}\left(\boldsymbol{a}^\top \boldsymbol{\xi}\right) \geq \sqrt{\frac{\gamma \ln n}{n}}\right\} \cdot n^{(2 + \delta + \delta^2/2) \gamma} \sqrt{\ln n}=t(\gamma, \delta)>0,
$$
where $t(\gamma, \delta)$ depends on $\gamma$ and $\delta$. Finally, letting $c(\gamma, \delta)=\min \{t(\gamma, \delta), 1 /(2 \sqrt{32 \gamma \pi})\}$ proves the lemma.

\section*{Appendix C: Proof of Proposition \ref{theorem.upper bound}}\label{app.C}
 To prove Proposition \ref{theorem.upper bound}, we need the following auxiliary lemma and corollary.
\setcounter{lemma}{0}
\renewcommand\thelemma{B.1}
 \begin{lemma}{\rm (\cite{kolda2001orthogonal})}\label{le.numerical lemma}
    Let $\mathcal{T} = \sum_{i=1}^{r}\lambda_i x_i\otimes y_i \otimes z_i$, $\|x_i\| = \|y_i\|= \|z_i\| =1$, $i = 1,2,\cdots,r$, and $(x_i^Tx_j)(y_i^Ty_j) = z_i^Tz_j = 0$ for $i \neq j$. The optimization problem
    \[
        \begin{array}{cc}
            \max & \mathcal{T}(x,y,z) \\
            \text{s.t.} & \|x\| = 1, \|y\| = 1, \|z\| = 1.\\
        \end{array}
    \]
    has the optimal value $v^* = \max{\lambda_i}$, where $\lambda_i \in \mR$, $\mathcal{T}\in \mR^{n_1 \times n_2 \times n_3}$, $x_i\in \mR^{n_1}$, $y_i\in \mR^{n_2}$, $z_i\in \mR^{n_3}$.
\end{lemma}
With the above lemma, we can get the following corollary.
\setcounter{corollary}{0}
\renewcommand\thecorollary{B.1}
\begin{corollary}\label{coro.example}
    For the all-ones tensor $\textbf{1}_{n_1 \times n_2 \times n_3}$, the optimization problem:
    \begin{equation}
    \begin{array}{cc}
    \max & \textbf{1}_{n_1 \times n_2 \times n_3}(x,y,z)\\
     \text{s.t.} & \|x\|^2=1,\|y\|^2=1,\|z\|^2=1
     \end{array}
    \end{equation}
    and 
    \begin{equation}
    \begin{array}{cc}
    \max & -\textbf{1}_{n_1 \times n_2 \times n_3}(x,y,z)\\
     \text{s.t.} & \|x\|^2=1,\|y\|^2=1,\|z\|^2=1
     \end{array}
    \end{equation}
has the optimal value $v^* = \sqrt{lmn}$, where $x\in \mR^l$, $y\in \mR^m$, $z\in \mR^n$.
\end{corollary}
\begin{proof}
    Let $\textbf{1}_{n_i}\in \mR^{n_i}(i=1,2,3)$ is the all-ones vector, it follows that
    $$
        \textbf{1}_{n_1 \times n_2 \times n_3}=\textbf{1}_{n_1}\otimes \textbf{1}_{n_2} \otimes \textbf{1}_{n_3} = \sqrt{n_1n_2n_3}(\textbf{1}_{n_1}/\sqrt{n_1})\otimes(\textbf{1}_{n_2}/\sqrt{n_2})\otimes(\textbf{1}_{n_3}/\sqrt{n_3})
    $$
    satisfies the conditions in Lemma \ref{le.numerical lemma}. Then the corollary holds.
\end{proof}
\textbf{Proof of Proposition \ref{theorem.upper bound}:}

With the conditions of Proposition \ref{theorem.upper bound} and the quaternion multiplication rules, problem (F) can be equivalently converted into
\begin{equation}\label{problem.simplified}
    \begin{array}{cc}
    \max & \mathcal{T}_0(x_0,y_0,z_0)- \mathcal{T}_0(x_1.y_1,z_0)+\mathcal{T}_0(x_2,y_2,z_0)- \mathcal{T}_0(x_3,y_3,z_0)\\
     & -\mathcal{T}_0(x_1,y_0,z_1)-\mathcal{T}_0(x_0,y_1,z_1)-\mathcal{T}_0(x_3,y_2,z_1)-\mathcal{T}_0(x_2,y_3,z_1)\\
     & + \mathcal{T}_0(x_2,y_0,z_2)-\mathcal{T}_0(x_3,y_1,z_2)+\mathcal{T}_0(x_0,y_2,z_2)-\mathcal{T}_0(x_1,y_3,z_2)\\
     & -\mathcal{T}_0(x_3,y_0,z_3)-\mathcal{T}_0(x_2,y_1,z_3)-\mathcal{T}_0(x_1,y_2,z_3)-\mathcal{T}_0(x_0,y_3,z_3)\\
     \text{s.t.} & \|x_0\|^2+\|x_1\|^2+\|x_2\|^2+\|x_3\|^2=1\\
     & \|y_0\|^2+\|y_1\|^2+\|y_2\|^2+\|y_3\|^2=1\\
     & \|z_0\|^2+\|z_1\|^2+\|z_2\|^2+\|z_3\|^2=1
     \end{array}
\end{equation}
where $\mathcal{T}_0=\textbf{1}_{n_1 \times n_2 \times n_3}$.
From Corollary \ref{coro.example}, it is easy to know the optimal value of the following problem (\ref{problem.upper}) is the upper bound for problem (\ref{problem.simplified}). 
\begin{equation}\label{problem.upper}
    \begin{array}{cc}
    \max & \sqrt{n_1n_2n_3}\big(\|x_0\|\|y_0\|\|z_0\|+ \|x_1\|\|y_1\|\|z_0\|+\|x_2\|\|y_2\|\|z_0\|+ \|x_3\|\|y_3\|\|z_0\|\\
     & +\|x_1\|\|y_0\|\|z_1\|+\|x_0\|\|y_1\|\|z_1\|+\|x_3\|\|y_2\|\|z_1\|+\|x_2\|\|y_3\|\|z_1\|\\
     & + \|x_2\|\|y_0\|\|z_2\|+\|x_3\|\|y_1\|\|z_2\|+\|x_0\|\|y_2\|\|z_2\|+\|x_1\|\|y_3\|\|z_2\|\\
     & +
     \|x_3\|\|y_0\|\|z_3\|+\|x_2\|\|y_1\|\|z_3\|+\|x_1\|\|y_2\|\|z_3\|+\|x_0\|\|y_3\|\|z_3\|\big)\\
     \text{s.t.} & \|x_0\|^2+\|x_1\|^2+\|x_2\|^2+\|x_3\|^2=1\\
     & \|y_0\|^2+\|y_1\|^2+\|y_2\|^2+\|y_3\|^2=1\\
     & \|z_0\|^2+\|z_1\|^2+\|z_2\|^2+\|z_3\|^2=1
     \end{array}
\end{equation}
and the optimal value of problem (\ref{problem.upper}) can be easily obtained which is $v_{\text{upper}} = 2\sqrt{n_1n_2n_3}$.

\bibliographystyle{abbrvnat}
\bibliography{main}
\end{document}